\documentclass[12pt]{article}%
\usepackage[colorlinks,bookmarks=true]{hyperref}
\usepackage{amsmath}
\usepackage{graphicx}
\usepackage{amsfonts}
\usepackage{amssymb}%
\providecommand{\U}[1]{\protect\rule{.1in}{.1in}}
\newtheorem{theorem}{Theorem} [section]

\newtheorem{corollary}{Corollary}[section]

\newtheorem{definition}{Definition} [section]

\newtheorem{lemma}{Lemma}[section]

\newtheorem{remark}{Remark} [section]

\newenvironment{proof}[1][Proof]{\textbf{#1.} }{\ \rule{1em}{1em}}
\numberwithin{equation}{section}

\begin{document}

\title{Nonlinear Modulational Instability of Dispersive PDE Models}
\author{Jiayin Jin, Shasha Liao, and Zhiwu Lin\\School of Mathematics\\Georgia Institute of Technology\\Atlanta, GA 30332, USA}
\maketitle

\begin{abstract}
We prove nonlinear modulational instability for both periodic and localized
perturbations of periodic traveling waves for several dispersive PDEs,
including the KDV type equations (e.g. the Whitham equation, the generalized
KDV equation, the Benjamin-Ono equation), the nonlinear Schr\"{o}dinger
equation and the BBM equation. First, the semigroup estimates required for the
nonlinear proof are obtained by using the Hamiltonian structures of the
linearized PDEs; Second, for KDV type equations the loss of derivative in the
nonlinear term is overcome in two complementary cases: (1) for smooth
nonlinear terms and general dispersive operators, we construct higher order
approximation solutions and then use energy type estimates; (2) for nonlinear
terms of low regularity, with some additional assumption on the dispersive
operator, we use a bootstrap argument to overcome the loss of derivative.

\end{abstract}

\section{Introduction\newline}

The modulational instability, also called Benjamin-Feir or side-band
instability in the literature, is a very important instability mechanism in
lots of dispersive and fluid models. It has been used to explain the
instability of periodic wave trains to self modulation and the development of
large-scale structures such as envelope solitons. The modulational instability
has been observed in experiments and in nature, for many physical systems. The
first theoretical understanding of modulational instability arose in 1960s, in
the works of Benjamin and Feir (\cite{benjamin-feir}) for water waves and
independently by Lighthill (\cite{lighthill}), Whitham (\cite{whitham-1967}),
Zakharov (\cite{zakharov}) for various dispersive wave equations. We refer to
the review (\cite{zakharov-review-2009}) for more details on the history and
physical applications of modulational instability. In recent years, there have
been lots of mathematical work on the rigorous justification of linear
modulational instability for various dispersive wave models including the KDV
type equations, the nonlinear Schr\"{o}dinger equation, the BBM equation etc.
In particular, the modulational instability conditions for perturbations of
long wavelength (i.e. frequencies near zero) were derived in lots of works
(\cite{bronski-hur-14} \cite{bronski-johnson-arma} \cite{haragus-BBM}
\cite{haragus-gallay-schrodinger} \cite{haragus-kapitula}
\cite{hur-johnson-whitham} \cite{hur-pandey} \cite{johnson-fractional-13}). We
refer to the recent survey (\cite{hur-et-survey-15}) for more details and
references. The modulational instability for perturbations of high frequencies
(i.e. not near zero) was also considered in some papers
(\cite{deconinck-high-frequency} \cite{hur-pandey-high-frequency}). However,
there has been no proof of modulational instability under the nonlinear
dynamics of the PDE models. The purpose of this paper is to provide a proof of
nonlinear modulational instability under both multi-periodic and localized
perturbations, for a large class of dispersive wave models.

We mainly consider the KDV type equations,
\begin{equation}
\partial_{t}{u}+\partial_{x}(\mathcal{M}u+f(u))=0,\label{kdvtp}%
\end{equation}
where $\mathcal{M}$ is a Fourier multiplier operator satisfying
$\widehat{\mathcal{M}u}(\xi)=\alpha(\xi)\widehat{u}(\xi)$ and $f(s)\in
C^{1}(\mathbf{R},\mathbf{R})$. We make the following assumptions on the
operator $\mathcal{M}$:

(A1) $\mathcal{M}$ is a self-adjoint operator, and the symbol $\alpha
:\mathbf{R}\mapsto\mathbf{R}^{+}$ is even and regular near $0$.

(A2) There exist constants $m,c_{1},c_{2}>0$, such that%

\begin{equation}
\text{(A2a) \ }c_{1}\left\vert \xi\right\vert ^{m}\leq\alpha\left(
\xi\right)  \leq c_{2}\left\vert \xi\right\vert ^{m},\ \text{for large }\xi,
\label{symbol-KDV}%
\end{equation}
or
\begin{equation}
\text{(A2b) \ }c_{1}\left\vert \xi\right\vert ^{-m}\leq\alpha\left(
\xi\right)  \leq c_{2}\left\vert \xi\right\vert ^{-m},\ \text{for large }\xi.
\label{symbol-whitham}%
\end{equation}

The assumption (\ref{symbol-KDV}) implies that $\mathcal{M}$ is an
\textquotedblleft differential\textquotedblright\ operator with $\Vert
\mathcal{M}(\cdot)\Vert_{L^{2}}\sim\Vert\cdot\Vert_{H^{m}}$, and
(\ref{symbol-whitham}) implies that $\mathcal{M}$ is a \textquotedblleft
smoothing\textquotedblright\ operator with $\Vert\mathcal{M}(\cdot
)\Vert_{H^{m}}\sim\Vert\cdot\Vert_{L^{2}}$.$\ $For the classical KDV equation,
$\mathcal{M=-\partial}_{x}^{2}$ and $f\left(  u\right)  =u^{2}$. Other
examples include: Benjamin-Ono equation, Whitham equation and intermediate
long-wave (ILW) equation, which are all of KDV type with $\alpha
(\xi)=\left\vert \xi\right\vert ,$ $\sqrt{\frac{\tanh\xi}{\xi}}$ and $\xi
\coth\left(  \xi H\right)  -H^{-1}$ respectively.

For convenience, we assume $\min\alpha(\xi)>0$. Since otherwise, we can always
break $\mathcal{M=M}_{1}+c_{1}$, where $\mathcal{M}_{1}$ has a positive symbol
and $c_{1}$ is a constant. Then in the traveling frame $\left(  x-c_{1}%
t,t\right)  $, the equation (\ref{kdvtp}) becomes
\[
\partial_{t}{u}+\partial_{x}(\mathcal{M}_{1}u+f(u))=0.
\]

A periodic traveling wave (TW) of (\ref{kdvtp}) is of the form $u\left(
x,t\right)  =u_{c}\left(  x-ct\right)  $, where $c\in\mathbf{R}$ is the
traveling speed and $u_{c}$ satisfies the equation
\begin{equation}
\mathcal{M}u_{c}-cu_{c}+f\left(  u_{c}\right)  =a,\label{eqn-TW-kdv-type}%
\end{equation}
for a constant $a$. The existence of the periodic TWs of
(\ref{eqn-TW-kdv-type}) had been well studied in the literature, and we refer
to the book (\cite{angulo-book}) and the references therein. In general, the
periodic TWs are a three-parameter family of solutions depending on period
$T$, traveling speed $c$ and the constant $a$. The stability of TWs to
perturbations of the same period has been studied a lot in recent years (e.g.
\cite{angulo-book} \cite{angulo-bona-et-06} \cite{hur-johnson-stability}
\cite{lin-zeng-Hamiltonian} \cite{johnson-fractional-13}). The modulational
instability is to study the instability of periodic TWs for perturbations of
different periods and even for localized perturbations in $\mathbf{R}$. The
equation (\ref{kdvtp}) in the traveling frame $\left(  x-ct,t\right)  $
becomes
\begin{equation}
\partial_{t}U-c\partial_{x}U+\partial_{x}(\mathcal{M}U+f(U))=0.\label{kdvtv}%
\end{equation}
The linearized equation of (\ref{kdvtv}) near $u_{c}$ can be written in the
Hamiltonian form
\begin{equation}
\partial_{t}U=-\partial_{x}\left(  \mathcal{M-}c+f^{\prime}(u_{c})\right)
U=JLU,\label{kdvln}%
\end{equation}
where
\begin{equation}
J=-\partial_{x},\ \ L=\mathcal{M-}c+f^{\prime}(u_{c}).\label{definition-J-L}%
\end{equation}
Without lost of generality, we take the minimal period $T=2\pi$. By the
standard Floquet-Bloch theory, any bounded eigenfunction $\phi(x)$ of $JL$
takes the form $\phi(x)=e^{ikx}v_{k}(x)$, where $k\in\left[  0,1\right]  $ is
a parameter and $v_{k}\in L^{2}(\mathbb{T}_{2\pi})$. It leads us to the
one-parameter family of eigenvalue problems
\[
JLe^{ikx}v_{k}(x)=\lambda(k)e^{ikx}v_{k}(x),
\]
or equivalently $J_{k}L_{k}v_{k}=\lambda\left(  k\right)  v_{k}$, where
\begin{equation}
J_{k}=\partial_{x}+ik,\ L_{k}=\mathcal{M}_{k}\mathcal{-}c+f^{\prime}%
(u_{c}).\label{defn-Jk-Lk}%
\end{equation}
Here, $\mathcal{M}_{k}$ is the Fourier multiplier operator with the symbol
$\alpha(\xi+k)$.

\begin{definition}
\label{defn-MI}We say that $u_{c}$ is linearly modulationally unstable if
there exists $k\in\left[  0,1\right]  $ such that the operator $J_{k}L_{k}$
has an unstable eigenvalue $\lambda(k)$ with $\operatorname{Re}\lambda(k)>0$
in the space $L^{2}(\mathbb{T}_{2\pi})$.
\end{definition}

By above definition and the analytic perturbation theory of the spectra of
linear operators, if $k_{0}$ is an unstable frequency, then all $k$ near
$k_{0}$ are also unstable. So there exist intervals of unstable frequencies in
$\left[  0,1\right]  $. For periodic waves which are orbitally stable under
co-periodic perturbations (i.e. same period), it is shown in Proposition 11.3
of \cite{lin-zeng-Hamiltonian} that when $k$ is small (i.e. long wavelength),
the possible unstable eigenvalues of $J_{k}L_{k}$ can only be perturbed from
the zero eigenvalue of $JL$ in $L^{2}\left(  \mathbb{T}_{2\pi}\right)  $. The
conditions of linear modulational instability for such long wavelength
perturbations had been studied in lots of papers for various dispersive models
(see the references cited at the beginning). In Section \ref{application}, we
give some examples for which the linear modulational instability condition is satisfied.

Our first main result is the proof of nonlinear modulational instability under
both multi-periodic and localized perturbations, for a smooth nonlinear term
$f\left(  u\right)  $ and $\mathcal{M}$ with a general symbol.

\begin{theorem}
\label{thm-smooth-f} Assume (A1)-(A2a) or (A1)-(A2b) and
(\ref{assumption-TW-whitham}), $f\in C^{\infty}\left(  \mathbf{R}\right)
\ $and $u_{c}$ is linearly modulationally unstable. Then

i) $u_{c}$ is nonlinearly orbitally unstable to (\ref{kdvtv}) for
multi-periodic perturbations in the following sense: there exists
$q\in\mathbb{N},$ $\theta_{0}>0$, such that for any $s\in\mathbb{N}$ and
arbitrarily small $\delta>0$, there exists a solution $U_{\delta}(t,x)$ to
(\ref{kdvtv}) satisfying $\Vert U_{\delta}(0,x)-u_{c}(x)\Vert_{H^{s}%
({\mathbb{T}_{2\pi q}})}<\delta\ $and
\[
\inf_{y\in\mathbb{T}}\Vert U_{\delta}(T^{\delta},x)-u_{c}(x+y)\Vert
_{L^{2}({\mathbb{T}_{2\pi q}})}\geqslant\theta_{0},
\]
where $T^{\delta}\sim\left\vert \ln\delta\right\vert $.

ii) $u_{c}$ is nonlinearly unstable to (\ref{kdvtv}) for localized
perturbations in the following sense: there exists $\theta_{0}>0$, such that
for any $s\in\mathbb{N}$ and arbitrarily small $\delta>0$, there exists a
solution $U_{\delta}(t,x)$ to (\ref{kdvtv}) satisfying $\Vert U_{\delta
}(0,x)-u_{c}(x)\Vert_{H^{s}(\mathbf{R})}<\delta\ $and$\ \Vert U(T^{\delta
},x)-u_{c}(x)\Vert_{L^{2}(\mathbf{R})}\geqslant\theta_{0}$, where $T^{\delta
}\sim|\ln\delta|$.
\end{theorem}

For some examples, $f\left(  u\right)  $ is not smooth. Our second result is
complementary to Theorem \ref{thm-smooth-f}, about nonlinear modulational
instability for non-smooth $f$ with some additional assumptions.

\begin{theorem}
\label{thm-bootstrap-instability}Assume
\begin{equation}
f\in C^{2n+2}\left(  \mathbf{R}\right)  ,\ \text{where\ }n\geqslant\frac{1}%
{2}\max\{1+m,1\}\ \text{is an integer,}\label{assumption-f}%
\end{equation}
the symbol $\alpha\left(  \xi\right)  \ $of $\mathcal{M}$ satisfies the
condition
\begin{equation}
c_{1}\left\vert \xi\right\vert ^{m}\leq\alpha\left(  \xi\right)  \leq
c_{2}\left\vert \xi\right\vert ^{m},\ \ m\geq1,c_{1},c_{2}>0,\ \text{for large
}\xi.\label{assumption-symbol-bootstrap}%
\end{equation}
Suppose $u_{c}$ is linearly modulationally unstable. Then $u_{c}$ is
nonlinearly unstable to (\ref{kdvtv}) for both multi-periodic and localized
perturbations in the sense of Theorem \ref{thm-smooth-f}, with the initial
perturbation arbitrarily small in $H^{2n}\left(  {\mathbb{T}_{2\pi q}}\right)
$ or $H^{2n}\left(  \mathbf{R}\right)  $.
\end{theorem}

\begin{remark}
In Theorem \ref{thm-bootstrap-instability}, the regularity assumption
(\ref{assumption-f}) on $f$ is only used to prove that the equation
(\ref{kdvtv}) is locally well-posed in $H^{2n}\left(  {\mathbb{T}_{2\pi q}%
}\right)  $ and $u_{c}+H^{2n}\left(  \mathbf{R}\right)  $ by Kato's approach
(see Lemma \ref{lemma-wp}). Assuming the local well-posedness of (\ref{kdvtv})
in the energy space $H^{\frac{m}{2}}$, we only need the following much weaker
assumptions on $f$ to prove nonlinear instability:

$f\in C^{1}\left(  \mathbf{R}\right)  \ $and there exist $p_{1}>1,\ p_{2}%
>2,\ $such that%
\begin{equation}
\left\vert f\left(  u+v\right)  -f\left(  v\right)  -f^{\prime}\left(
v\right)  u\right\vert \leq C\left(  \left\vert u\right\vert _{\infty
},\left\vert v\right\vert _{\infty}\right)  \left\vert u\right\vert ^{p_{1}%
},\label{assumption-f-bootstrap}%
\end{equation}%
\begin{equation}
\left\vert F\left(  u+v\right)  -F\left(  v\right)  -f\left(  v\right)
u-\frac{1}{2}f^{\prime}\left(  v\right)  u^{2}\right\vert \leq C\left(
\left\vert u\right\vert _{\infty},\left\vert v\right\vert _{\infty}\right)
\left\vert u\right\vert ^{p_{2}},\label{assumption-F-bootstrap}%
\end{equation}
where $F\left(  u\right)  =\int_{0}^{u}f\left(  s\right)  ds$. The conditions
(\ref{assumption-f-bootstrap})-(\ref{assumption-F-bootstrap}) are
automatically satisfied when $f\in C^{2}\left(  \mathbf{R}\right)  $.
\end{remark}

In above Theorems, the nonlinear instability for multi-periodic perturbations
is proved in the orbital distance since (\ref{kdvtv}) is translation
invariant. For localized perturbations, we study the equation (\ref{kdvtv}) in
the space $u_{c}+H^{s}\left(  \mathbf{R}\right)  $ which is not translation
invariant. Therefore, we do no use the orbital distance for nonlinear
instability under localized perturbations.

Below we discuss main ingredients in the proof of Theorems \ref{thm-smooth-f}
and \ref{thm-bootstrap-instability}. For the proof of nonlinear instability,
first we need to establish the semigroup estimates for the linearized equation
(\ref{kdvln}), more specifically, to show that the growth of solutions of
(\ref{kdvln}) is bounded by the maximal growth rate of unstable eigenvalues of
the linearized operator. To get such estimates, we strongly use the
Hamiltonian structure of the linearized equation (\ref{kdvln}). For
multi-periodic perturbations, since $L$ has only finitely many negative modes,
this fits into the general theory developed by Lin and Zeng
(\cite{lin-zeng-Hamiltonian}) and the semigroup estimates follow directly from
the exponential trichotomy Theorem \ref{theorem-dichotomy}. For localized
perturbations, the quadratic form of $L$ has infinitely many negative modes
and we cannot use Theorem \ref{theorem-dichotomy} directly. By observing that
any function $u\in H^{s}{(\mathbf{R)}}$ can be written as
\[
u(x)=\int_{0}^{1}e^{i\xi x}u_{\xi}(x)d\xi,\ \text{where\ }u_{\xi}%
(x)=\Sigma_{n\in\mathbf{Z}}e^{inx}\hat{u}(n+\xi)\in H^{s}(\mathbb{T}_{2\pi}),
\]
and
\begin{equation}
\Vert u(x)\Vert_{H^{s}{(\mathbf{R})}}^{2}\thickapprox\int_{0}^{1}\Vert u_{\xi
}\left(  x\right)  \Vert_{H^{s}(\mathbb{T}_{2\pi})}^{2}\,d\xi
,\ \label{norm-local-equivalence}%
\end{equation}%
\[
e^{tJL}u(x)=\int_{0}^{1}e^{i\xi x}e^{tJ_{\xi}L_{\xi}}u_{\xi}\left(  x\right)
\,d\xi,
\]
the estimate of $e^{tJL}|_{H^{s}{(\mathbf{R})}}$ is reduced to estimate
$e^{tJ_{\xi}L_{\xi}}|_{H^{s}(\mathbb{T}_{2\pi})}$ uniformly for $\xi\in\left[
0,1\right]  $. This is proved in \cite{lin-zeng-Hamiltonian} for the case when
$\mathcal{M}$ is \textquotedblleft differential\textquotedblright\ (i.e.
(\ref{symbol-KDV})) and in Lemma \ref{lemma-semigroup-whitham-localized} when
$\mathcal{M}$ is \textquotedblleft smoothing\textquotedblright\ (i.e.
(\ref{symbol-whitham})).

With the semigroup estimates, to prove nonlinear instability we still need to
overcome the loss of derivative of the nonlinear term in (\ref{kdvtv}). We use
two different approaches to handle two complementary cases. For the case of
smooth nonlinear term and general $\mathcal{M}$ including both
\textquotedblleft differential\textquotedblright\ and \textquotedblleft
smoothing\textquotedblright\ cases, we basically adapt Grenier's approach in
\cite{grenier-2000} developed for proving nonlinear instability of shear flows
of the 2D Euler equation. The idea is to construct higher order approximate
solutions of (\ref{kdvtv}) and then use the energy estimates to overcome the
loss of derivative. When the nonlinear term is smooth, the approximate
solution can be constructed to sufficiently high order to compensate for the
roughness of the energy estimates. For the multi-periodic case, the initial
perturbation is chosen to be along the most unstable eigenfunction. For the
localized case, since there is no genuine eigenfunction of $JL$ in
$L^{2}\left(  \mathbf{R}\right)  $, the initial perturbation is constructed as
a wave packet concentrated near the most unstable frequency.

When the nonlinear term is not smooth, we cannot use the approach of higher
order approximate solutions. Instead, a totally different approach of
bootstrap estimates is used to overcome the loss of derivative when $f$ is
$C^{1}$ with the growth conditions (\ref{assumption-f-bootstrap}%
)-(\ref{assumption-F-bootstrap}) and $\mathcal{M}$ is \textquotedblleft
differential\textquotedblright\ with the condition (\ref{symbol-KDV}). First,
the invariance of the energy functional is used to show that $H^{\frac{m}{2}}%
$norm of the unstable solution has the same growth as the $L^{2}$ norm. Then
we estimate the growth of $H^{-1}$ norm with the help of the semigroup
estimate $e^{tJL}|_{H^{-1}}$. The estimates are closed by interpolating
$H^{-1}$ and $H^{\frac{m}{2}}$ to get back to $L^{2}$. The loss of derivative
in the nonlinear term $\partial_{x}f\left(  u\right)  \ $is overcome by
observing that
\[
\left\Vert \partial_{x}f\left(  u\right)  \right\Vert _{H^{-1}}\thickapprox
\left\Vert f\left(  u\right)  \right\Vert _{L^{2}},
\]
which is controllable in $H^{\frac{m}{2}}$. To get the crucial semigroup
estimate $e^{tJL}|_{H^{-1}}$ used in the above bootstrap process, by duality
it is equivalently to estimate $e^{tLJ}|_{H^{1}}$, which is then related to
$e^{tJL}$ by certain conjugate transforms. The proof is much more involved for
the localized case. By using the norm equivalence
(\ref{norm-local-equivalence}), this is reduced to estimate $e^{tL_{\xi}%
J_{\xi}}|_{H^{1}(\mathbb{T}_{2\pi})}$ uniformly for $\xi\in\left[  0,1\right]
$. This is done by a careful decomposition of the spectral projections of
$L_{\xi}$ near $0$ and away from $0$. We note that the idea of overcoming the
loss of derivative by bootstrapping the growth of higher order norms from a
lower order one was originated in \cite{guo-strauss95} for the Vlasov-Poisson
system. This approach was later extended to treat other problems including 2D
Euler equation (\cite{guo-bardos02} \cite{lin-euler-imrn}) and Vlasov-Maxwell
systems (\cite{lin-vm-nonlinear}). Here, our approach of bootstrapping the
lower order norm $\left(  H^{-1}\right)  \ $from a higher order norm $\left(
H^{\frac{m}{2}}\right)  $ and then closing by interpolation seems to be new.
This idea coupled with the $H^{-1}$ semigroup estimates could be useful in
other problems involving the loss of derivative.

Besides the KDV type equations, modulational instability also appears in
semilinear models such as BBM and Schr\"{o}dinger equations. Since there is no
loss of derivative, the nonlinear instability can be proved by ODE arguments.
The required semigroup estimates can be obtained similarly by using the
Hamiltonian structures. As an example, we consider BBM equation in Section 7.

This paper is organized as follows. In Section 2, we study the regularity of
unstable eigenfunctions. In Section 3, we gather and prove the semigroup
estimates used in the proof of nonlinear instability. In Sections 4 and 5, the
nonlinear instability for multi-periodic and localized cases is proved by
constructing higher order approximate solutions. In Section 6, we prove
nonlinear instability by bootstrap arguments. In Section 7, we prove nonlinear
instability for BBM equation. In the final Section 8, we list some models for
which our theorems are applicable.

\section{Linear Modulational Instability}

In this section, we prove the regularity of the unstable eigenfunctions of
$J_{k}L_{k}$. In the proof below and throughout this paper, we use $g\lesssim
h$ $\left(  g,h\geq0\right)  \ $to denote $g\leq Ch$, $\ $for a generic
constant $C>0$, which may differ from one inequality to another. First, we
consider the case when $\mathcal{M}$ is a "differential" operator as in the
case of the KDV, the Benjamin-Ono and the ILW equations.

\begin{lemma}
\label{lem-eigenfunction-kdv} Assume (\ref{symbol-KDV}). If \thinspace$f\in
C^{\infty}\left(  \mathbf{R}\right)  $ and $v_{k}(x)\in L^{2}(\mathbb{T}%
_{2\pi})$ is an unstable eigenfunction to $J_{k}L_{k}$ with $k\in\left[
0,1\right]  $, then $v_{k}\in H^{s}(\mathbb{T}_{2\pi})$ for every
$s\in\mathbb{N}$.
\end{lemma}

\begin{proof}
By assumption, there exists $\lambda\left(  k\right)  $ with
$\operatorname{Re}\lambda\left(  k\right)  >0$ such that
\begin{equation}
J_{k}L_{k}v_{k}=\lambda\left(  k\right)  v_{k},\ \ v_{k}(x)\in L^{2}%
(\mathbb{T}_{2\pi}),\label{egv-Jk-Lk}%
\end{equation}
where $J_{k},L_{k}$ are defined in (\ref{defn-Jk-Lk}). It is easy to see that
$J_{k}$ is a skew-adjoint operator and $L_{k}$ is a self-adjoint operator.
Taking the real part of the $L^{2}$ inner product of (\ref{egv-Jk-Lk}) with
$L_{k}v_{k}$, we get the following \textquotedblleft conservation
law\textquotedblright:
\[
\operatorname{Re}\left\langle \lambda v_{k},L_{k}v_{k}\right\rangle
=\operatorname{Re}\left\langle J_{k}L_{k}v_{k},L_{k}v_{k}\right\rangle =0.
\]
Since $L_{k}$ is self-adjoint, $\left\langle v_{k}(x),L_{k}v_{k}%
(x)\right\rangle $ is real. It follows that
\[
(\operatorname{Re}\lambda)\left\langle v_{k},L_{k}v_{k}\right\rangle =0.
\]
Noting that $\operatorname{Re}\lambda>0$, we have
\[
\left\langle v_{k},L_{k}v_{k}\right\rangle =\left\langle v_{k}%
(x),(c-\mathcal{M}_{k}-f^{\prime}(u_{c}))v_{k}(x)\right\rangle =0,
\]
i.e.
\[
c\int_{\mathbb{T}_{2\pi}}v_{k}(x)\overline{v_{k}(x)}\,dx-\int_{\mathbb{T}%
_{2\pi}}v_{k}(x)\overline{\mathcal{M}_{k}v_{k}(x)}\,dx-\int_{\mathbb{T}_{2\pi
}}v_{k}(x)\overline{f^{\prime}(u_{c})v_{k}(x)}\,dx=0.
\]
It follows immediately that%
\begin{align*}
|\langle\mathcal{M}_{k}v_{k},v_{k}\rangle| &  \leq|c\int_{\mathbb{T}_{2\pi}%
}v_{k}(x)\overline{v_{k}(x)}\,dx|+|\int_{\mathbb{T}_{2\pi}}v_{k}%
(x)\overline{f^{\prime}(u_{c})v_{k}(x)}\,dx|\\
&  \leq(c+\Vert f^{\prime}(u_{c})\Vert_{L^{\infty}(\mathbb{T}_{2\pi})})\Vert
v_{k}(x)\Vert_{L^{2}(\mathbb{T}_{2\pi})}^{2},
\end{align*}
Applying $\mathcal{M}_{k}\ $to (\ref{egv-Jk-Lk}), we obtain
\begin{equation}
\mathcal{M}_{k}J_{k}L_{k}v_{k}=\lambda\mathcal{M}_{k}v_{k}.\label{egv2}%
\end{equation}
Taking the real part of inner product of (\ref{egv2}) with $L_{k}v_{k}$, one
has
\[
\operatorname{Re}\langle\mathcal{M}_{k}J_{k}L_{k}v_{k},L_{k}v_{k}%
\rangle=\operatorname{Re}\langle\lambda\mathcal{M}_{k}v_{k},L_{k}v_{k}\rangle.
\]
Note that $\mathcal{M}_{k}$ is self-adjoint and $J_{k}$ is skew-adjoint, also
$\mathcal{M}_{k}$ and $J_{k}$ are commutable, therefore $\mathcal{M}_{k}J_{k}$
is skew-adjoint. It follows that
\[
\operatorname{Re}\langle\mathcal{M}_{k}J_{k}L_{k}v_{k},L_{k}v_{k}\rangle=0,
\]
which implies
\[
\operatorname{Re}\langle\lambda\mathcal{M}_{k}v_{k},L_{k}v_{k}\rangle=0.
\]
Then, we obtain
\[
(\operatorname{Re}\lambda)\langle\mathcal{M}_{k}v_{k},\mathcal{M}_{k}%
v_{k}\rangle=\operatorname{Re}\langle\lambda\mathcal{M}_{k}v_{k},\left(
c-f^{\prime}(u_{c})\right)  v_{k}\rangle,
\]
which implies that
\[
\Vert\mathcal{M}_{k}v_{k}\Vert_{L^{2}(\mathbb{T}_{2\pi})}^{2}\lesssim(c+\Vert
f^{\prime}(u_{c})\Vert_{C^{\left[  \frac{m}{2}\right]  +1}(\mathbb{T}_{2\pi}%
)})\Vert\mathcal{M}_{k}^{\frac{1}{2}}v_{k}\Vert_{L^{2}(\mathbb{T}_{2\pi})}%
^{2},
\]
and
\[
\Vert v_{k}\Vert_{H^{m}(\mathbb{T}_{2\pi})}\lesssim|\langle\mathcal{M}%
_{k}v_{k},v_{k}\rangle|\lesssim\Vert v_{k}\Vert_{L^{2}(\mathbb{T}_{2\pi}%
)}.\newline%
\]
In the above, we use the estimate
\begin{align*}
\left\vert \langle\mathcal{M}_{k}v_{k},f^{\prime}(u_{c})v_{k}\rangle
\right\vert  &  \leq\left\Vert \mathcal{M}_{k}^{\frac{1}{2}}\left(  f^{\prime
}(u_{c})v_{k}\right)  \right\Vert _{L^{2}}\Vert\mathcal{M}_{k}^{\frac{1}{2}%
}v_{k}\Vert_{L^{2}}\\
&  \leq\left\Vert f^{\prime}(u_{c})v_{k}\right\Vert _{H^{\frac{m}{2}%
}(\mathbb{T}_{2\pi})}\Vert\mathcal{M}_{k}^{\frac{1}{2}}v_{k}\Vert_{L^{2}}\\
&  \leq\Vert f^{\prime}(u_{c})\Vert_{C^{\left[  \frac{m}{2}\right]
+1}(\mathbb{T}_{2\pi})}\left\Vert v_{k}\right\Vert _{H^{\frac{m}{2}%
}(\mathbb{T}_{2\pi})}\Vert\mathcal{M}_{k}^{\frac{1}{2}}v_{k}\Vert_{L^{2}}\\
&  \leq\Vert f^{\prime}(u_{c})\Vert_{C^{\left[  \frac{m}{2}\right]
+1}(\mathbb{T}_{2\pi})}\Vert\mathcal{M}_{k}^{\frac{1}{2}}v_{k}\Vert
_{L^{2}(\mathbb{T}_{2\pi})}^{2}.
\end{align*}

Similarly, one can show that
\[
\Vert v_{k}\Vert_{H^{s}(\mathbb{T}_{2\pi})}\leq C(s)\Vert v_{k}\Vert
_{L^{2}(\mathbb{T}_{2\pi})}%
\]
for any $s\in\mathbb{N}$.
\end{proof}

In the next Lemma, we prove the regularity of unstable eigenfunctions when
$\mathcal{M}$ is a smoothing operator satisfying (\ref{symbol-whitham}) as in
the case of Whitham equation. We need the following assumption on the periodic
TWs of (\ref{kdvtp}):
\begin{equation}
c-\Vert f^{\prime}(u_{c})\Vert_{L^{\infty}(\mathbb{T}_{2\pi})}\geqslant
\delta_{0}>0, \label{assumption-TW-whitham}%
\end{equation}
which was assumed in \cite{hur-johnson-whitham} and \cite{solitary-whitham} to
show the regularity of TWs of Whitham equation. This assumption is satisfied
for small amplitude waves of Whitham equation (see Section
\ref{subsection-whitham}). By a similar proof as in \cite{hur-johnson-whitham}
and \cite{solitary-whitham}, we can show that $u_{c}\in C^{\infty}$ under the
assumption (\ref{assumption-TW-whitham}) when $f\in C^{\infty}\left(
\mathbf{R}\right)  .$

\begin{lemma}
\label{lem-eigenfunction-whitham} Assume (\ref{symbol-whitham}) and
(\ref{assumption-TW-whitham}). If \thinspace$f\in C^{\infty}\left(
\mathbf{R}\right)  $ and $v_{k}\in L^{2}(\mathbb{T}_{2\pi})$ is an unstable
eigenfunction of $J_{k}L_{k}$ with $k\in\left[  0,1\right]  $, then $v_{k}\in
H^{s}(\mathbb{T}_{2\pi})$ for every $s\in\mathbb{N}$.
\end{lemma}

\begin{proof}
Step 1: We first prove that under the assumption (\ref{assumption-TW-whitham}%
), for any integer $s\geq0$, there exists a constant $C(s)$, such that for any
$\phi\in H^{s}(\mathbb{T}_{2\pi}),$
\begin{equation}
\Vert(\lambda+D)^{-1}D\phi\Vert_{H^{s}(\mathbb{T}_{2\pi})}\leq C(s)\Vert
\phi\Vert_{H^{s}(\mathbb{T}_{2\pi})},\label{claim-bound-D}%
\end{equation}
where $D=-(\partial_{x}+ik)\left(  c-f^{\prime}(u_{c})\right)  $.

Define an inner product $[\cdot,\cdot]$ by
\[
\lbrack u,v]=\langle u,(c-f^{\prime}(u_{c}))v\rangle.
\]
One can check that
\[
\lbrack Du,v]=-[u,Dv],
\]
i.e. $D$ is skew-adjoint in the inner product $[\cdot,\cdot]$. For any $u\in
Dom(D)=H^{1}(\mathbb{T}_{2\pi})$, denote $v=(\lambda+D)u\in L^{2}%
(\mathbb{T}_{2\pi})$, then one has
\[
|[v,u]|\geq|\operatorname{Re}[(\lambda+D)u,u]|=(\operatorname{Re}\lambda)[u,u]
\]
Also, by the Schwartz inequality, one has
\[
|[v,u]|\leq\lbrack v,v]^{1/2}[u,u]^{1/2}.
\]
It follows that
\[
\Vert(\lambda+D)u\Vert_{\bar{L}^{2}(\mathbb{T}_{2\pi})}\geq(\operatorname{Re}%
\lambda)\Vert u\Vert_{\bar{L}^{2}(\mathbb{T}_{2\pi})},
\]
where $\Vert\cdot\Vert_{{\bar{L}^{2}(\mathbb{T}_{2\pi})}}:=[\cdot,\cdot
]^{1/2}$. Note that $c-\Vert f^{\prime}(u_{c})\Vert_{L^{\infty}(\mathbb{T}%
_{2\pi})}\geq\delta_{0}>0$, so $\bar{L}^{2}(\mathbb{T}_{2\pi})\sim
L^{2}(\mathbb{T}_{2\pi})$. Thus, $\lambda+D$ is invertible and $(\lambda
+D)^{-1}$ is bounded from $L^{2}(\mathbb{T}_{2\pi})$ to $L^{2}(\mathbb{T}%
_{2\pi})$. Taking the inner product of the equation $v=(\lambda+D)u$ with
$Du$, we obtain
\[
\lbrack v,Du]=\lambda\lbrack u,Du]+[Du,Du],
\]
which implies that $\Vert Du\Vert_{L^{2}(\mathbb{T}_{2\pi})}\leq C\Vert
v\Vert_{L^{2}(\mathbb{T}_{2\pi})}$ for some constant $C$. Hence, we have that
$(\lambda+D)^{-1}$ is bounded from $L^{2}(\mathbb{T}_{2\pi})$ to
$H^{1}(\mathbb{T}_{2\pi})$, from which it follows immediately that
$(\lambda+D)^{-1}D$ is bounded from $H^{1}(\mathbb{T}_{2\pi})$ to
$H^{1}(\mathbb{T}_{2\pi})$. Also, $(\lambda+D)^{-1}D=I-\lambda(\lambda
+D)^{-1}$ is bounded from $L^{2}$ to $L^{2}$.

We now prove (\ref{claim-bound-D}) by induction. Suppose it is true for $0\leq
s\leq l$. Let $\psi=(\lambda+D)^{-1}D\phi$, then
\[
D\phi=(\lambda+D)\psi.
\]
From
\[
\partial_{x}^{l}\lambda\psi+\partial_{x}^{l}D\psi=\partial_{x}^{l}D\phi,
\]
we get
\[
\partial_{x}^{l}\psi=(\lambda+D)^{-1}[\partial_{x}^{l}D\phi+(D\partial_{x}%
^{l}-\partial_{x}^{l}D)\psi].
\]
It is easy to check the following commutator estimate
\[
\Vert(D\partial_{x}^{l}-\partial_{x}^{l}D)\psi\Vert_{L^{2}}\leq C(l)\Vert
\psi\Vert_{H^{l}}.
\]
Therefore,
\[
\Vert(\lambda+D)^{-1}(D\partial_{x}^{l}-\partial_{x}^{l}D)\psi\Vert_{H^{1}%
}\leq C(s)\Vert\psi\Vert_{H^{l}}.
\]
Also,
\begin{align*}
&  \Vert(\lambda+D)^{-1}\partial_{x}^{l}D\phi\Vert_{H^{1}}\\
\leq &  \Vert(\lambda+D)^{-1}D\partial_{x}^{l}\phi\Vert_{H^{1}}+\Vert
(\lambda+D)^{-1}\left(  \partial_{x}^{l}D-D\partial_{x}^{l}\right)  \phi
\Vert_{H^{1}}\\
\leq &  C(s)\left(  \Vert\partial_{x}^{l}\phi\Vert_{H^{1}}+\Vert\left(
\partial_{x}^{l}D-D\partial_{x}^{l}\right)  \phi\Vert_{L^{2}}\right)  \\
\leq &  C(s)\Vert\phi\Vert_{H^{l+1}}.
\end{align*}
Thus, we obtain
\[
\Vert\partial_{x}^{l}\psi\Vert_{H^{1}}\leq C(s)(\Vert\phi\Vert_{H^{l+1}}%
+\Vert\psi\Vert_{H^{l}})\leq C\left(  s\right)  \Vert\phi\Vert_{H^{l+1}},
\]
by the induction assumption. This finishes the proof of (\ref{claim-bound-D}).

Step 2: From (\ref{egv2}), one has
\begin{equation}
v_{k}(x)=(\lambda+D)^{-1}D\left(  c-f^{\prime}(u_{c})\right)  ^{-1}%
\mathcal{M}_{k}v_{k}(x). \label{vmv}%
\end{equation}
Let $B=(\lambda+D)^{-1}D(c-f^{\prime}(u_{c}))^{-1}$, then we have
$v_{k}(x)=B\mathcal{M}_{k}v_{k}(x)$. Since $f\left(  u_{c}\right)  \in
C^{\infty}$ and $c-f^{\prime}(u_{c})\geqslant\delta_{0}>0$, by
(\ref{claim-bound-D}) $B$ is bounded from $H^{s}(\mathbb{T}_{2\pi})$ to
$H^{s}(\mathbb{T}_{2\pi})$, for any integer $s\geq0$. By using the
interpolation theory, $B$ is bounded from $H^{s}(\mathbb{T}_{2\pi})$ to
$H^{s}(\mathbb{T}_{2\pi})$ for any $s\geq0$. So we have
\[
\Vert B\mathcal{M}_{k}v_{k}(x)\Vert_{H^{m}(\mathbb{T}_{2\pi})}\leq C\left\Vert
\mathcal{M}_{k}v_{k}(x)\right\Vert _{H^{m}(\mathbb{T}_{2\pi})}\leq C\Vert
v_{k}(x)\Vert_{L^{2}(\mathbb{T}_{2\pi})}.
\]
Repeatedly using the identity $v_{k}(x)=B\mathcal{M}_{k}v_{k}(x)$, we arrive
at
\[
v_{k}(x)=B\mathcal{M}_{k}v_{k}(x)=B\mathcal{M}_{k}B\mathcal{M}_{k}%
v_{k}(x)=\cdots=(B\mathcal{M}_{k})^{n}v_{k}(x),
\]
which implies that
\[
\Vert v_{k}(x)\Vert_{H^{nm}(\mathbb{T}_{2\pi})}\leq C(n)\Vert v_{k}%
(x)\Vert_{L^{2}(\mathbb{T}_{2\pi})}.
\]
Since $n\in\mathbb{N}$ is arbitrary, this finishes the proof of the lemma.
\end{proof}

\section{Semigroup estimates}

In this section, we consider semigroup estimates for $e^{tJL}$ (equivalently,
for the linearized equation (\ref{kdvln})), which will be used in later
sections to prove nonlinear instability. First, we consider the estimates in
both multi-periodic space $H^{s}\left(  {\mathbb{T}_{2\pi q}}\right)  \ $and
localized spaces $H^{s}\left(  \mathbf{R}\right)  $, with $s\geq\frac{m}{2}$.
Such estimates are given in Section 11.4 of \cite{lin-zeng-Hamiltonian} and we
only sketch it here. It is obtained by using the theory in
(\cite{lin-zeng-Hamiltonian}) for general linear Hamiltonian PDEs which we
describe below. Consider a linear Hamiltonian system
\begin{equation}
\partial_{t}u=JLu,\ u\in X,\nonumber
\end{equation}
where $X$ is a Hilbert space. Assume that:

\textbf{(H1)} $J:X^{\ast}\rightarrow X$ is a skew-adjoint operator.

\textbf{(H2)} The operator $L:X\rightarrow X^{\ast}$ generates a bounded
bilinear symmetric form $\left\langle L\cdot,\cdot\right\rangle $ on $X$%
.$\ $There exists a decomposition $X=X_{-}\oplus\ker L\oplus X_{+}$ satisfying
that $\left\langle L\cdot,\cdot\right\rangle |_{X_{-}}<0,\ \dim X_{-}%
=n^{-}\left(  L\right)  <\infty$, and there exists $\delta_{1}>0$ such that%

\[
\left\langle Lu,u\right\rangle \geq\delta_{1}\left\Vert u\right\Vert _{X}%
^{2}\ ,\text{ for any }u\in X_{+}.
\]

\textbf{(H3)} The above $X_{\pm}$ satisfy
\[
\ker i_{X_{+}\oplus X_{-}}^{\ast}=\{f\in X^{\ast}\mid\langle f,u\rangle
=0,\,\forall u\in X_{-}\oplus X_{+}\}\subset D(J),
\]
where $i_{X_{+}\oplus X_{-}}^{\ast}:X^{\ast}\rightarrow(X_{+}\oplus
X_{-})^{\ast}$ is the dual operator of the embedding $i_{X_{+}\oplus X_{-}}$.

The assumption \textbf{(H3)} is automatically satisfied when $\dim\ker
L<\infty$, as in the case of this paper.

\begin{theorem}
\cite{lin-zeng-Hamiltonian}\label{theorem-dichotomy} Under assumptions
(\textbf{H1})-(\textbf{H3}), $JL$ generates a $C^{0}$ group $e^{tJL}$ of
bounded linear operators on $X$ and there exists a decomposition%
\[
X=E^{u}\oplus E^{c}\oplus E^{s},\quad\dim E^{u}=\dim E^{s}\leq n^{-}(L)
\]
satisfying: \newline i) $E^{u},E^{s}$ and $E^{c}$ are invariant under
$e^{tJL}$; \newline ii) $E^{c}=\{u\in X\mid\langle Lu,v\rangle=0,\ \forall
v\in E^{s}\oplus E^{u}\}$;\newline iii) let $\lambda_{u}=\max
\{\operatorname{Re}\lambda\mid\lambda\in\sigma(JL|_{E^{u}})\}$, there exist
$M>0$ and an integer $k_{0}\geq0$, such that
\begin{equation}%
\begin{split}
&  \left\vert e^{tJL}|_{E^{s}}\right\vert _{X}\leq M(1+t^{\dim E^{s}%
-1})e^{-\lambda_{u}t},\quad\forall\;t\geq0;\ \\
&  |e^{tJL}|_{E^{u}}|_{X}\leq M(1+|t|^{\dim E^{u}-1})e^{\lambda_{u}t}%
,\quad\forall\;t\leq0,
\end{split}
\label{estimate-stable.unstable}%
\end{equation}%
\begin{equation}
\ |e^{tJL}|_{E^{c}}|_{X}\leq M(1+\left\vert t\right\vert ^{k_{0}}%
),\quad\forall\;t\in\mathbf{R}, \label{estimate-center}%
\end{equation}
and
\[
k_{0}\leq1+2\left(  n^{-}(L)-\dim E^{u}\right)  ;
\]
Moreover, for $k\geq1$, define the space $X^{k}\subset X$ to be
\[
X^{k}=D\left(  (JL)^{k}\right)  =\left\{  u\in X\ |\ \left(  JL\right)
^{n}u\in X,\ n=1,\cdots,k.\right\}
\]
and
\begin{equation}
\left\Vert u\right\Vert _{X^{k}}=\left\Vert u\right\Vert +\left\Vert
JLu\right\Vert +\cdots+\Vert(JL)^{k}u\Vert. \label{norm-X-k}%
\end{equation}
Assume $E^{u,s}\subset X^{k}$, then the exponential trichotomy for $X^{k}$
holds true: $X^{k}$ is decomposed as a direct sum%
\[
X^{k}=E^{u}\oplus E_{k}^{c}\oplus E^{s},\ E_{k}^{c}=E^{c}\cap X^{k}%
\]
and the estimates (\ref{estimate-stable.unstable}) and (\ref{estimate-center})
still hold in the norm $X^{k}$.
\end{theorem}

By using above Theorem, we can prove the following estimates for the
linearized equation (\ref{kdvln}).

\begin{lemma}
\label{lemma-trichotomy}Consider the semigroup $e^{tJL}$ associated with the
solutions of (\ref{kdvln}), where $J,L$ are given in (\ref{definition-J-L}).

i) (KDV type) Assume (\ref{symbol-KDV}), the exponential trichotomy in the
sense of (\ref{estimate-stable.unstable}) and (\ref{estimate-center}) holds
true in the spaces $H^{s}\left(  {\mathbb{T}_{2\pi q}}\right)  $ $\left(
s\geq\frac{m}{2},q\in\mathbb{N}\right)  $

ii) (Whitham type) Assume (\ref{symbol-whitham}) and
(\ref{assumption-TW-whitham}), then the exponential trichotomy of $e^{tJL}$
holds true in the spaces $H^{s}\left(  {\mathbb{T}_{2\pi q}}\right)  $
$\left(  s\geq0,q\in\mathbb{N}\right)  .$
\end{lemma}

\begin{proof}
It suffices to check the assumption (H2) in Theorem \ref{theorem-dichotomy},
since (H1) is obvious and (H3) is automatic due to $\dim\ker L<\infty$.

For i), the quadratic form $\left\langle L\cdot,\cdot\right\rangle $ is
bounded in the space $H^{\frac{m}{2}}({\mathbb{T}_{2\pi q}})$. The operator
$L$ is a compact perturbation of $\mathcal{M}$, whose spectrum in
$H^{m}\left(  {\mathbb{T}_{2\pi q}}\right)  $ are positive and discrete.
Therefore, $L$ has at most a finite number of negative eigenvalues, that is,
$n^{-}\left(  L\right)  <\infty$. Thus, the exponential trichotomy of
$e^{tJL}$ is true in $H^{\frac{m}{2}}({\mathbb{T}_{2\pi q}})$. By the proof of
Lemma \ref{lem-eigenfunction-kdv}, any stable or unstable eigenfunction of
$JL$ in $L^{2}({\mathbb{T}_{2\pi q}})$ lies in $H^{s}\left(  {\mathbb{T}_{2\pi
q}}\right)  $ for any $s>0$. Therefore, the exponential trichotomy of
$e^{tJL}$ is also true in $H^{s}\left(  {\mathbb{T}_{2\pi q}}\right)  $ for
any $s\geq\frac{m}{2}$.

For ii), the quadratic form $\left\langle L\cdot,\cdot\right\rangle $ is
bounded in the space $L^{2}({\mathbb{T}_{2\pi q}})$. Under the condition
(\ref{assumption-TW-whitham}), the operator $-L$ is a compact perturbation of
the positive operator $c-f^{\prime}\left(  u_{c}\right)  ,$ thus $n^{-}\left(
-L\right)  <\infty$. Applying Theorem \ref{theorem-dichotomy} to $JL=\left(
-J\right)  \left(  -L\right)  $, we get the exponential trichotomy of
$e^{tJL}$ in $L^{2}({\mathbb{T}_{2\pi q}})$ and in $H^{s}\left(
{\mathbb{T}_{2\pi q}}\right)  $ $\left(  s\geq0\right)  $ by the regularity of
stable and unstable eigenfunctions of $JL$ in $L^{2}({\mathbb{T}_{2\pi q}})$
as in Lemma \ref{lem-eigenfunction-whitham}.
\end{proof}

As an immediate corollary of the above lemma, we get the following upper bound
on the growth of the semigroup $e^{tJL}$.

\begin{corollary}
\label{cor-upperbound-semigroup}Let $\lambda_{0}$ be the growth rate of the
most unstable eigenvalue of $JL$. Then under the conditions of Lemma
\ref{lemma-trichotomy}, for any $\varepsilon>0$, there exists constant
$C_{\varepsilon}$ such that
\[
\left\Vert e^{tJL}\right\Vert _{H^{s}\left(  {\mathbb{T}_{2\pi q}}\right)
}\leq C_{\varepsilon}e^{\left(  \lambda_{0}+\varepsilon\right)  t},\text{ for
any }t>0,
\]
where $q\in\mathbb{N}$ , $s\geq s_{0}$ with $s_{0}=\frac{m}{2}$ for case i)
and $s_{0}=0$ for case ii).
\end{corollary}

The above semigroup estimate implies the following lemma for the inhomogeneous
equation. For convenience, we use $\mathbb{T}$ for $\mathbb{T}_{2\pi q}$.

\begin{lemma}
\label{semigestim} If $\Vert g\left(  t\right)  \Vert_{H^{s}(\mathbb{T}%
)}\leqslant C_{g}e^{wt}$, for some $s\geq s_{0}$ and $w>\lambda_{0}$, then the
solution to the equation
\[
\partial_{t}u=JLu+g,\text{ }u|_{t=0}=0,
\]
satisfies
\[
\Vert u\Vert_{H^{s}(\mathbb{T})}\lesssim C_{g}e^{wt}.
\]

\end{lemma}

\begin{proof}
Using Corollary \ref{cor-upperbound-semigroup} with $\epsilon=\frac{1}%
{2}\left(  w-\lambda_{0}\right)  $, we have
\[
\left\Vert e^{tJL}\right\Vert _{H^{s}}\lesssim e^{\frac{1}{2}\left(
\lambda_{0}+w\right)  t}.
\]
Then
\begin{align*}
\left\Vert u\left(  t\right)  \right\Vert _{H^{s}(\mathbb{T})}  &  =\left\Vert
\int_{0}^{t}e^{\left(  t-s\right)  JL}g\left(  s\right)  ds\right\Vert
_{H^{s}(\mathbb{T})}\\
&  \lesssim\int_{0}^{t}e^{\frac{1}{2}\left(  \lambda_{0}+w\right)  \left(
t-s\right)  }C_{g}e^{ws}ds\leq C_{g}e^{wt}\frac{2}{w-\lambda_{0}}.
\end{align*}

\end{proof}

To prove nonlinear instability for localized perturbations, we need to study
the semigroup $e^{tJL}$ on the space $H^{s}\left(  \mathbf{R}\right)  $
$\left(  s\geq\frac{m}{2}\right)  $. In general, the operator $L$ has negative
continuous spectra in $H^{s}\left(  \mathbf{R}\right)  $. For example, when
$\mathcal{M=-\partial}_{x}^{2}$, the spectrum of $L=$ $\mathcal{-\partial}%
_{x}^{2}$ $+V\left(  x\right)  $ with a periodic potential $V\left(  x\right)
$ is well studied in the literature and is known to have bands of continuous
spectrum. So Theorem \ref{theorem-dichotomy} does not apply. However, we can
prove upper bound estimate of $e^{tJL}$ on $H^{s}\left(  \mathbf{R}\right)  $,
which suffices for proving nonlinear localized instability. We will need the
following lemma to estimate $e^{tJL}$ on $H^{s}\left(  \mathbf{R}\right)  $.

\begin{lemma}
\label{ptor} Suppose $h(k,x)\in H_{x}^{s}(\mathbb{T})$ for any $k\in I$, where
$I$ is a measurable set contained in an interval with length less than or
equal to $1$, then $\int_{I}h(k,x)e^{ikx}\,dk\in H_{x}^{s}(\mathbf{R})$ if and
only if $\Vert h(k,x)\Vert_{H_{x}^{s}(\mathbb{T})}\in L_{k}^{2}\left(
I\right)  $. More precisely, there exist constants $C_{1}(s),C_{2}(s)>0$, such
that
\[
\Vert\int_{I}h(k,x)e^{ikx}\,dk\Vert_{H_{x}^{s}(\mathbf{R})}^{2}\geq
C_{1}(s)\int_{I}\Vert h(k,x)\Vert_{H_{x}^{s}(\mathbb{T})}^{2}\,dk,
\]
and
\[
\Vert\int_{I}h(k,x)e^{ikx}\,dk\Vert_{H_{x}^{s}(\mathbf{R})}^{2}\leq
C_{2}(s)\int_{I}\Vert h(k,x)\Vert_{H_{x}^{s}(\mathbb{T})}^{2}\,dk.
\]

\end{lemma}

\begin{proof}
First, we write $h(k,x)$ as a Fourier series
\[
h(k,x)=\sum_{j\in\mathbb{Z}}\widehat{h}(k,j)e^{ijx}.
\]
By direct computations, we have
\begin{align*}
&  \ \ \ \ \ \partial_{x}^{s}\int_{I}h(k,x)e^{ikx}\,dk\\
&  =\int_{I}\sum_{j\in\mathbb{Z}}(i(k+j))^{s}\widehat{h}(k,j)e^{i(k+j)x}%
\,dk\newline=\sum_{j\in\mathbb{Z}}\int_{I_{j}}(ik)^{s}\widehat{h}%
(k-j,j)e^{ikx}\,dk\newline\\
&  =\sum_{j\in\mathbb{Z}}\int_{\mathbf{R}}(ik)^{s}\mathcal{X}_{I_{j}%
}(k)\widehat{h}(k-j,j)e^{ikx}\,dk\newline=\int_{\mathbf{R}}\left(  \sum
_{j\in\mathbb{Z}}(ik)^{s}\mathcal{X}_{I_{j}}(k)\widehat{h}(k-j,j)\right)
e^{ikx}\,dk\newline\\
&  =\left(  \sum_{j\in\mathbb{Z}}(ik)^{s}\mathcal{X}_{I_{j}}(k)\widehat{h}%
(k-j,j)\right)  ^{\vee}(x),
\end{align*}
where
\[
I_{j}=I+j,\ \mathcal{X}_{I_{j}}(k)=%
\begin{cases}
1 & \text{ if $k\in I_{j}$}\\
0 & \text{ if $k\notin I_{j}$}%
\end{cases}
.\newline%
\]
Note that $I\ $is contained in an interval with length no more than $1$,
therefore $I_{j}\ $are disjoint with each other, which implies that
$\mathcal{X}_{I_{j_{1}}}\mathcal{X}_{I_{j_{2}}}=0\ $almost everywhere, if
$j_{1}\neq j_{2}$. Then by Parseval's identity, we have
\begin{align*}
&  \ \ \ \ \ \ \Vert\partial_{x}^{s}\int_{I}h(k,x)e^{ikx}\,dk\Vert_{L_{x}%
^{2}(\mathbf{R})}^{2}\\
&  =\Vert\sum_{j\in\mathbb{Z}}(ik)^{s}\mathcal{X}_{I_{j}}(k)\widehat{h}%
(k-j,j)\Vert_{L_{k}^{2}(\mathbf{R})}^{2}\newline=\int_{\mathbf{R}}\sum
_{j\in\mathbb{Z}}|\mathcal{X}_{I_{j}}(k)|^{2}\left\vert k\right\vert
^{2s}|\widehat{h}(k-j,j)|^{2}\,dk\newline\\
&  =\sum_{j\in\mathbb{Z}}\int_{I_{j}}\left\vert k\right\vert ^{2s}%
|\widehat{h}(k-j,j)|^{2}\,dk\newline=\int_{I}\sum_{j\in\mathbb{Z}}\left\vert
k+j\right\vert ^{2s}|\widehat{h}(k,j)|^{2}\,dk\\
&  \thickapprox\int_{I}\sum_{j\in\mathbb{Z}}\left\vert j\right\vert
^{2s}|\widehat{h}(k,j)|^{2}\,dk.
\end{align*}
Then the desired results follow directly.
\end{proof}

Now, we are ready to prove the upper bound estimate of $e^{tJL}$ on
$H^{s}\left(  \mathbf{R}\right)  $. The following semigroup estimates were
proved in \cite{lin-zeng-Hamiltonian} for the \textquotedblleft
differential\textquotedblright\ case (\ref{symbol-KDV}).

\begin{lemma}
\cite{lin-zeng-Hamiltonian}\label{lemma-semigroup-local-KDV}Assume
(\ref{symbol-KDV}) and
\begin{equation}
\limsup_{\left\vert \xi\right\vert \rightarrow\infty}\frac{\alpha^{\prime}%
(\xi)}{\left\vert \xi\right\vert ^{m}}<\infty. \label{assumption-decay-symbol}%
\end{equation}
Let $\lambda_{0}\geq0$ be such that
\begin{equation}
\operatorname{Re}\lambda\leq\lambda_{0},\quad\forall\xi\in\lbrack
0,1],\;\lambda\in\sigma(J_{\xi}L_{\xi}). \label{defn-largest-lambda-0}%
\end{equation}
Then for every $s\geq\frac{m}{2}$, $\varepsilon>0\ $there exist
$C(s,\varepsilon)>0$ such that
\[
\Vert e^{tJL}u(x)\Vert_{H^{s}{(\mathbf{R})}}\leqslant C(s,\varepsilon
)e^{\left(  \lambda_{0}+\varepsilon\right)  t}\Vert u(x)\Vert_{H^{s}%
{(\mathbf{R})}},\text{ }\forall t>0,
\]
for any $u\in H^{s}{(\mathbf{R})}$.
\end{lemma}

\begin{remark}
The assumption (\ref{assumption-decay-symbol}) can be replaced by a weaker
assumption
\[
\lim_{\rho\rightarrow0}\sup_{\xi\in\mathbf{Z}}\frac{|\alpha(\xi+\rho
)-\alpha(\xi)|}{1+|\xi|^{m}}\rightarrow0.
\]

\end{remark}

For the Whitham type equation, we have the following similar result.

\begin{lemma}
\label{lemma-semigroup-whitham-localized}Assume (\ref{symbol-whitham}) and
(\ref{assumption-TW-whitham}), then for every $s\geq0$, $\varepsilon>0\ $there
exists $C(s,\varepsilon)>0$ such that
\begin{equation}
\Vert e^{tJL}u(x)\Vert_{H^{s}{(\mathbf{R})}}\leqslant C(s,\varepsilon
)e^{\left(  \lambda_{0}+\varepsilon\right)  t}\Vert u(x)\Vert_{H^{s}%
{(\mathbf{R})}},\text{ }\forall t>0, \label{estimate-local-whitham}%
\end{equation}
for any $u\in H^{s}{(\mathbf{R})}$. Here, $\lambda_{0}$ is the largest growth
rate as defined in (\ref{defn-largest-lambda-0}).
\end{lemma}

\begin{proof}
The proof is similar to that of Lemma \ref{lemma-semigroup-local-KDV} (or
Lemma 11.2 in \cite{lin-zeng-Hamiltonian}). We sketch it here. First, for any
$u\in H^{s}{(\mathbf{R)}}$,
\[
u(x)=\int_{0}^{1}e^{i\xi x}u_{\xi}(x)d\xi,\;\text{ where }u_{\xi}%
(x)=\Sigma_{n\in\mathbf{Z}}e^{inx}\hat{u}(n+\xi)\in H^{s}(\mathbb{T}_{2\pi}),
\]
and $\hat{u}$ is the Fourier transform of $u$. By Lemma \ref{ptor}, there
exists $C>0$ such that
\begin{equation}
\frac{1}{C}\Vert u\Vert_{H^{s}(\mathbf{R})}^{2}\leq\int_{0}^{1}\Vert u_{\xi
}\left(  x\right)  \Vert_{H_{x}^{s}(\mathbb{T}_{2\pi})}^{2}\,d\xi\leq C\Vert
u\Vert_{H^{s}(\mathbf{R})}^{2}. \label{norm-equivalence}%
\end{equation}
Note that
\[
e^{tJL}u(x)=\int_{0}^{1}e^{i\xi x}e^{tJ_{\xi}L_{\xi}}u_{\xi}\left(  x\right)
\,d\xi,
\]
and thus
\begin{equation}
\Vert e^{tJL}u(x)\Vert_{H^{s}{(\mathbf{R})}}^{2}\thickapprox\int_{0}^{1}\Vert
e^{tJ_{\xi}L_{\xi}}u_{\xi}\left(  x\right)  \Vert_{H_{x}^{s}(\mathbb{T}_{2\pi
})}^{2}\,d\xi. \label{semigroup-frequency}%
\end{equation}

So to prove (\ref{estimate-local-whitham}), it suffices to show that: for any
$\varepsilon,s\geq0$, there exists $C(s,\varepsilon)>0$ such that
\begin{equation}
\Vert e^{tJ_{\xi}L_{\xi}}v(x)\Vert_{H_{x}^{s}(\mathbb{T}_{2\pi})}\leq
C(s,\varepsilon)e^{\left(  \lambda_{0}+\varepsilon\right)  t}\Vert
v(x)\Vert_{H_{x}^{s}(\mathbb{T}_{2\pi})},\ \forall\xi\in\lbrack
0,1].\label{semigroup-modu-whitham}%
\end{equation}
It suffices to prove the lemma for $s=0$ since the estimates for general
$s\geq0$ can be obtained by applying $J_{\xi}L_{\xi}$ repeatedly to the
estimates for $s=0$ (and interpolation for the case when $s$ is not an
integer). Due to the compactness of $[0,1]$, it suffices to prove that for any
$\xi_{0}\in\lbrack0,1]$, there exist $C,\epsilon>0$ such that
(\ref{semigroup-modu-whitham}) holds for $\xi\in(\xi_{0}-\epsilon,\xi
_{0}+\epsilon)$. We first note that each $\lambda\in\sigma(J_{\xi_{0}}%
L_{\xi_{0}})$ is an isolated eigenvalue with finite algebraic multiplicity and
$L_{\xi_{0}}$ is non-degenerate on $E_{\lambda}$ when $\lambda\neq0$ and on
$E_{0}/(E_{0}\cap\ker L_{\xi_{0}})$, where $E_{\lambda}$ is the generalized
eigenspace of the eigenvalue $\lambda$ of $J_{\xi_{0}}L_{\xi_{0}}$. By
(\ref{assumption-TW-whitham}), $n^{-}\left(  -L_{\xi}\right)  <\infty$. Let
\[
\Lambda=\{\lambda\in\sigma(J_{\xi_{0}}L_{\xi_{0}})\mid\exists\ \delta>0\text{
s.t. }\langle-L_{\xi_{0}}v,v\rangle\geq\delta\Vert v\Vert^{2}\text{ on
}E_{\lambda}\}.
\]
By the instability index formula (Proposition 11.2 in
\cite{lin-zeng-Hamiltonian}), $\sigma(J_{\xi_{0}}L_{\xi_{0}})\backslash
\Lambda$ is finite and thus
\[
n=\Sigma_{\lambda\in\sigma(J_{\xi_{0}}L_{\xi_{0}})\backslash\Lambda}\dim
E_{\lambda}<\infty.
\]
Moreover, there exists $\varepsilon_{0}>0$ such that $\Omega\cap
\Lambda=\emptyset$, where
\[
\Omega=\cup_{\lambda\in\sigma(J_{\xi_{0}}L_{\xi_{0}})\backslash\Lambda}%
\{z\mid\left\vert z-\lambda\right\vert <\varepsilon_{0}\}\subset\mathbf{C}.
\]
Assuming that:
\begin{equation}
\text{the resolvent }(\lambda-J_{\xi}L_{\xi})^{-1}\ \text{is continuous in
}\xi\in\lbrack0,1],\label{claim-resolvent-whitham}%
\end{equation}
we now prove (\ref{semigroup-modu-whitham}) for $\xi$ in a small interval near
$\xi_{0}$. Indeed, by (\ref{claim-resolvent-whitham}), there exists
$\epsilon>0$ such that $\partial\Omega\cap\sigma(J_{\xi}L_{\xi})=\emptyset$
for any $\xi\in\lbrack\xi_{0}-\epsilon,\xi_{0}+\epsilon]$. For such $\xi$,
let
\[
P(\xi)=\frac{1}{2\pi i}\oint_{\partial\Omega}(\lambda-J_{\xi}L_{\xi}%
)^{-1}d\lambda,\quad Z_{\xi}=P(\xi)X,\quad Y_{\xi}=\big(I-P(\xi)\big)X,
\]
which are continuous in $\xi$ and invariant under $e^{tJ_{\xi}L_{\xi}}$.
Therefore $\dim Z_{\xi}=n$. By the definition of $\Omega$, we know that
$-L_{\xi_{0}}|_{Y_{\xi_{0}}}$ is positive definite. Then the continuity of
$L_{\xi}$ in $\xi$ implies that there exists $\delta_{0}>0$ such that
\[
\delta_{0}^{-2}\Vert v\Vert^{2}\geq\langle-L_{\xi}v,v\rangle\geq\delta_{0}%
^{2}\Vert v\Vert^{2},\quad\forall v\in Y_{\xi},\;|\xi-\xi_{0}|\leq\epsilon.
\]
So for any $\xi\in\lbrack\xi_{0}-\epsilon,\xi_{0}+\epsilon]$, there exists a
generic constant $C>0$ independent of $\xi$, such that for any $v\in
L^{2}\left(  \mathbb{T}_{2\pi}\right)  $,
\begin{align*}
&  \Vert e^{tJ_{\xi}L_{\xi}}v\Vert\leq\Vert e^{tJ_{\xi}L_{\xi}}P(\xi
)v\Vert+\Vert e^{tJ_{\xi}L_{\xi}}\big(I-P(\xi)\big)v\Vert\\
\leq &  C\Big((1+t^{n})e^{\lambda_{0}t}\Vert P(\xi)v\Vert+\langle-L_{\xi
}e^{tJ_{\xi}L_{\xi}}\big(I-P(\xi)\big)v,e^{tJ_{\xi}L_{\xi}}\big(I-P(\xi
)\big)v\rangle^{\frac{1}{2}}\Big)\\
\leq &  C\Big((1+t^{n})e^{\lambda_{0}t}\Vert P(\xi)v\Vert+\langle-L_{\xi
}\big(I-P(\xi)\big)v,\big(I-P(\xi)\big)v\rangle^{\frac{1}{2}}\Big)\\
\leq &  C(1+t^{n})e^{\lambda_{0}t}\Vert v\Vert_{L^{2}}\leq C\left(
\varepsilon\right)  e^{\left(  \lambda_{0}+\varepsilon\right)  t}\Vert
v\Vert_{L^{2}}.
\end{align*}
Along with the compactness of $[0,1]$, this implies estimates
(\ref{semigroup-modu-whitham}) and (\ref{estimate-local-whitham}).

It remains to prove (\ref{claim-resolvent-whitham}) about the continuity of
the resolvent. Fix $k\in\lbrack0,1]$. For $k^{\prime}$ near $k$, we have
\[
J_{k^{\prime}}L_{k^{\prime}}-J_{k}L_{k}=(\partial_{x}+ik)(\mathcal{M}%
_{k^{\prime}}-\mathcal{M}_{k})+i(k^{\prime}-k)\left(  \mathcal{M}_{k^{\prime}%
}-c+f^{\prime}(u_{c})\right)  .
\]
Let $D=(\partial_{x}+ik)\left(  c-f^{\prime}(u_{c})\right)  $, then by
(\ref{assumption-TW-whitham}) and the proof of Lemma
\ref{lem-eigenfunction-whitham}, for any $a_{0}>0$, $\left(  a_{0}+D\right)
^{-1}:L^{2}\rightarrow H^{1}$ is bounded. So
\begin{equation}
\left\vert \left(  a_{0}+D\right)  ^{-1}\left(  J_{k^{\prime}}L_{k^{\prime}%
}-J_{k}L_{k}\right)  \right\vert _{L^{2}\rightarrow L^{2}}\rightarrow
0\ \text{as }k\rightarrow k^{\prime}. \label{converge-zero-whitham}%
\end{equation}
Moreover,
\[
I+\left(  a_{0}+D\right)  ^{-1}(\lambda-J_{k}L_{k})=\left(  a_{0}+D\right)
^{-1}\left(  \lambda+a_{0}-(\partial_{x}+ik)\mathcal{M}_{k}\right)
\]
is compact in $L^{2}$. Therefore $A=\left(  a_{0}+D\right)  ^{-1}%
(\lambda-J_{k}L_{k})$ is a Fredholm operator of index $0$. Suppose
$\lambda\notin\sigma(J_{k}L_{k})$, then $A$ is injective and thus $A^{-1}$ is
bounded on $L^{2}$. Along with (\ref{converge-zero-whitham}), we obtain
\[
|(\lambda-J_{k}L_{k})^{-1}(J_{k^{\prime}}L_{k^{\prime}}-J_{k}L_{k}%
)|=|A^{-1}\left(  a_{0}+D\right)  ^{-1}(J_{k^{\prime}}L_{k^{\prime}}%
-J_{k}L_{k})|\rightarrow0
\]
as $k^{\prime}\rightarrow k$. From
\[
\lambda-J_{k^{\prime}}L_{k^{\prime}}=(\lambda-J_{k}L_{k})\left(
I-(\lambda-J_{k}L_{k})^{-1}(J_{k^{\prime}}L_{k^{\prime}}-J_{k}L_{k})\right)
,
\]
we obtain the continuity of the resolvent $(\lambda-J_{k}L_{k})^{-1}$ in
$k\in\lbrack0,1]$. This finishes the proof of the lemma.
\end{proof}

The following is an analogue of Lemma \ref{semigestim}.

\begin{lemma}
\label{lemma-semigroup-inhomo-line}If
\[
\Vert g\left(  t\right)  \Vert_{H^{s}(\mathbf{R})}\leqslant C_{g}\frac{e^{wt}%
}{1+t^{b}},\ \ t\geq0
\]
for some $b>0$, $s\geq s_{0}$ and $w>\lambda_{0}$, then the solution to the
equation
\[
\partial_{t}u=JLu+g,\text{ }u|_{t=0}=0,
\]
satisfies
\[
\Vert u\Vert_{H^{s}(\mathbf{R})}\leqslant CC_{g}\frac{e^{wt}}{1+t^{b}}%
,\ t\geq0.
\]

\end{lemma}

\begin{proof}
Choose $\epsilon=\frac{1}{2}\left(  w-\lambda_{0}\right)  $ in Lemma
\ref{lemma-semigroup-whitham-localized}, then $\left\Vert e^{tJL}\right\Vert
_{H^{s}}\lesssim e^{\frac{1}{2}\left(  \lambda_{0}+w\right)  t}$. So we have
\begin{align*}
\left\Vert u\left(  t\right)  \right\Vert _{H^{s}}  &  \leq\int_{0}%
^{t}\left\Vert e^{\left(  t-s\right)  JL}\right\Vert _{H^{s}}\left\Vert
g\left(  s\right)  \right\Vert _{H^{s}}ds\\
&  \lesssim\int_{0}^{t}e^{\frac{1}{2}\left(  \lambda_{0}+w\right)  \left(
t-s\right)  }\frac{e^{ws}}{1+s^{b}}ds\\
&  =e^{wt}\int_{0}^{t}e^{\frac{1}{2}\left(  \lambda_{0}-w\right)  \left(
t-s\right)  }\frac{1}{1+s^{b}}ds\lesssim\frac{e^{wt}}{1+t^{b}},
\end{align*}
since $w>\lambda_{0}$.
\end{proof}

Lastly, we prove the semigroup estimates in the space $H^{-1}$, which will be
used in the proof of nonlinear instability by bootstrap arguments. First, we
consider the estimates for periodic perturbations.

\begin{lemma}
\label{lemma-semigroup-H-1}Consider the semigroup $e^{tJL}$ associated with
the solutions of (\ref{kdvln}), where $J,L$ are given in (\ref{definition-J-L}%
). Assume (\ref{symbol-KDV}) or (\ref{symbol-whitham}) and
(\ref{assumption-TW-whitham}), then for any $\varepsilon>0\ $there exist
$C(\varepsilon)>0$ such that
\[
\Vert e^{tJL}u(x)\Vert_{H^{-1}{(\mathbb{T}_{2\pi q})}}\leqslant C(\varepsilon
)e^{\left(  \lambda_{0}+\varepsilon\right)  t}\Vert u(x)\Vert_{H^{-1}%
{(\mathbb{T}_{2\pi q})}},\text{ }\forall t>0,
\]
for any $u\in H^{-1}{(\mathbb{T}_{2\pi q})}$.
\end{lemma}

\begin{proof}
Since $\left(  JL\right)  ^{\ast}=-LJ$, by duality it suffices to show that
\begin{equation}
\Vert e^{tLJ}\Vert_{H^{1}{(\mathbb{T}_{2\pi q})}}\leq C(\varepsilon)e^{\left(
\lambda_{0}+\varepsilon\right)  t},\text{ \ \ }\forall t>0.
\label{estimate-dual-H-1}%
\end{equation}
Denote $P^{0}$ and $P^{1}=1-P^{0}$ to be the projection operators to $\ker L$
and $\left(  \ker L\right)  ^{\perp}=R\left(  L\right)  $ respectively. For
any $v\in H^{1}{(\mathbb{T}_{2\pi q})}$, let $v=P^{0}v+P^{1}v=v_{1}+v_{2}.$
Then the equation $\partial_{t}v=LJv$ can be written as
\begin{equation}
\partial_{t}v_{1}=0,\ \ \ \partial_{t}v_{2}=LJv_{1}+LJv_{2}.
\label{eqn-dual-periodic}%
\end{equation}
Since $L_{1}=L|_{R\left(  L\right)  }:R\left(  L\right)  \rightarrow R\left(
L\right)  $ has a bounded inverse and
\[
LJ|_{R\left(  L\right)  }=L_{1}P^{1}JL_{1}L_{1}^{-1},\ \ e^{tLJ|_{R\left(
L\right)  }}=L_{1}P^{1}e^{tJL}|_{R\left(  L\right)  }L_{1}^{-1},
\]
by Lemma \ref{cor-upperbound-semigroup} we have
\[
\left\Vert e^{tLJ|_{R\left(  L\right)  }}\right\Vert _{H^{1}}\lesssim
\left\Vert e^{tJL}|_{R\left(  L\right)  }\right\Vert _{H^{1+m}}\leq
C(\varepsilon)e^{\left(  \lambda_{0}+\varepsilon\right)  t},\
\]
for the case of (\ref{symbol-KDV}), and
\[
\left\Vert e^{tLJ|_{R\left(  L\right)  }}\right\Vert _{H^{1}}\lesssim
\left\Vert e^{tJL}|_{R\left(  L\right)  }\right\Vert _{H^{1}}\leq
C(\varepsilon)e^{\left(  \lambda_{0}+\varepsilon\right)  t}%
\]
for the case of (\ref{symbol-whitham}) and (\ref{assumption-TW-whitham}). By
(\ref{eqn-dual-periodic}), we have $\left\Vert v_{1}\left(  t\right)
\right\Vert _{H^{1}}=\left\Vert v_{1}\left(  0\right)  \right\Vert _{H^{1}}$
and
\begin{align*}
\left\Vert v_{2}\left(  t\right)  \right\Vert _{H^{1}}  &  \leq\left\Vert
e^{tLJ|_{R\left(  L\right)  }}v_{2}\left(  0\right)  \right\Vert _{H^{1}}%
+\int_{0}^{t}\left\Vert e^{\left(  t-s\right)  LJ|_{R\left(  L\right)  }%
}LJv_{1}\left(  0\right)  \right\Vert _{H^{1}}ds\\
&  \lesssim C(\varepsilon)e^{\left(  \lambda_{0}+\varepsilon\right)  t}\left(
\left\Vert v_{2}\left(  0\right)  \right\Vert _{H^{1}}+\left\Vert v_{1}\left(
0\right)  \right\Vert _{H^{1}}\right) \\
&  \lesssim C(\varepsilon)e^{\left(  \lambda_{0}+\varepsilon\right)
t}\left\Vert v\left(  0\right)  \right\Vert _{H^{1}},
\end{align*}
which implies (\ref{estimate-dual-H-1}) and the lemma.
\end{proof}

In the next lemma, we consider localized perturbations.

\begin{lemma}
\label{lemma-localized-H-1}Consider the semigroup $e^{tJL}$ associated with
the solutions of (\ref{kdvln}), where $J,L$ are given in (\ref{definition-J-L}%
). Assume (\ref{symbol-KDV}) or (\ref{symbol-whitham}) and
(\ref{assumption-TW-whitham}), then for any $\varepsilon>0\ $there exist
$C(\varepsilon)>0$ such that
\[
\Vert e^{tJL}u(x)\Vert_{H^{-1}{(\mathbf{R})}}\leqslant C(\varepsilon
)e^{\left(  \lambda_{0}+\varepsilon\right)  t}\Vert u(x)\Vert_{H^{-1}%
{(}\mathbf{R}{)}},\text{ }\forall t>0,
\]
for any $u\in H^{-1}{(}\mathbf{R}{)}$.
\end{lemma}

\begin{proof}
By duality, it suffices to show that
\[
\Vert e^{tLJ}\Vert_{H^{1}{(\mathbf{R})}}\leqslant C(\varepsilon)e^{\left(
\lambda_{0}+\varepsilon\right)  t}.
\]
As in the proof of Lemmas \ref{lemma-semigroup-local-KDV} and
\ref{lemma-semigroup-whitham-localized}, it is enough to show that for any
$\varepsilon>0$, there exists $C\left(  \varepsilon\right)  >0$ such that
\begin{equation}
\Vert e^{tL_{\xi}J_{\xi}}u(x)\Vert_{H^{1}(\mathbb{T}_{2\pi})}\leq
C(\varepsilon)e^{\left(  \lambda_{0}+\varepsilon\right)  t}\Vert
u(x)\Vert_{H^{1}(\mathbb{T}_{2\pi})},\ \label{estimate-local-dual}%
\end{equation}
is true for any $\xi\in\lbrack0,1]$ and $u\in H^{1}(\mathbb{T}_{2\pi})$. By
compactness of $\left[  0,1\right]  $, again it suffices to prove that for any
$\xi_{0}\in\lbrack0,1]$, there exist $C,\epsilon>0$ such that
(\ref{estimate-local-dual}) holds for $\xi\in(\xi_{0}-\epsilon,\xi
_{0}+\epsilon)$. We consider two cases below.

Case 1 ($L_{\xi_{0}}$ is invertible): In this case, there exists $\epsilon>0$
such that $L_{\xi}$ is invertible for $\xi\in(\xi_{0}-\epsilon,\xi
_{0}+\epsilon)$. So we have
\[
\left\Vert e^{tL_{\xi}J_{\xi}}\right\Vert _{H^{1}(\mathbb{T}_{2\pi}%
)}=\left\Vert L_{\xi}e^{tJ_{\xi}L_{\xi}}L_{\xi}^{-1}\right\Vert _{H^{1}%
(\mathbb{T}_{2\pi})}\lesssim\left\Vert e^{tJ_{\xi}L_{\xi}}\right\Vert
_{H^{1+m}(\mathbb{T}_{2\pi})}\leq C(\varepsilon)e^{\left(  \lambda
_{0}+\varepsilon\right)  t}%
\]
for the case of (\ref{symbol-KDV}), and
\[
\left\Vert e^{tL_{\xi}J_{\xi}}\right\Vert _{H^{1}(\mathbb{T}_{2\pi}%
)}=\left\Vert L_{\xi}e^{tJ_{\xi}L_{\xi}}L_{\xi}^{-1}\right\Vert _{H^{1}%
(\mathbb{T}_{2\pi})}\lesssim\left\Vert e^{tJ_{\xi}L_{\xi}}\right\Vert
_{H^{1}(\mathbb{T}_{2\pi})}\leq C(\varepsilon)e^{\left(  \lambda
_{0}+\varepsilon\right)  t}%
\]
for the case of (\ref{symbol-whitham}) and (\ref{assumption-TW-whitham}). In
the above, we use the estimate (\ref{semigroup-modu-whitham}) which is true
for both cases of (\ref{symbol-KDV}) and (\ref{symbol-whitham}%
)-(\ref{assumption-TW-whitham}).

Case 2 ($L_{\xi_{0}}$ is not invertible): In this case, $\ker L_{\xi_{0}}%
\neq\left\{  0\right\}  $. It is possible that $L_{\xi}$ is invertible for
$\xi$ near $\xi_{0}$. For example, when $\mathcal{M=-\partial}_{x}^{2}$, it
was shown in Remark 11.1 of \cite{lin-zeng-Hamiltonian} that $L_{\xi}$ has
zero eigenvalue if and only if $\xi=0,1$. However, for $\xi$ near $\xi_{0}$,
there is no uniform (in $\xi$) estimate for $L_{\xi}^{-1}$ and we cannot argue
as in Case 1. We will separate the eigenspaces of $L_{\xi}$ ($\xi\ $%
near\ $\xi_{0}$) for eigenvalues near $0$ and away $0$. Since $0$ is an
isolated eigenvalue of $L_{\xi_{0}}$, so
\[
d_{0}=\min\left\{  \left\vert \lambda\right\vert ,\ \lambda\in\sigma\left(
L_{\xi_{0}}\right)  /\left\{  0\right\}  \right\}  >0\text{.}%
\]
Let $\epsilon>0$ be small enough such that when $\xi\in(\xi_{0}-\epsilon
,\xi_{0}+\epsilon)$,
\[
\Gamma=\left\{  z\ |\ \left\vert z\right\vert =\frac{d_{0}}{2}\right\}
\cap\sigma(L_{\xi})=\emptyset.
\]
Denote $P_{\xi}^{0}=\oint_{\Gamma}\left(  z-L_{\xi}\right)  ^{-1}dz$ to be the
Riesz projection associated with the eigenvalues of $L_{\xi}$ inside $\Gamma,$
and $P_{\xi}^{1}=1-P_{\xi}^{0}$. In particular, $P_{\xi_{0}}^{0},\ P_{\xi_{0}%
}^{1}$ are the projection operators to $\ker L_{\xi_{0}}$ and $R\left(
L_{\xi_{0}}\right)  $ respectively. By choosing $\epsilon$ small, we can
assume that: $\dim R\left(  P_{\xi}^{0}\right)  =\dim\ker L_{\xi_{0}},$%
\[
\min\left\{  \left\vert \lambda\right\vert ,\ \lambda\in\sigma\left(  L_{\xi
}|_{R\left(  P_{\xi}^{1}\right)  }\right)  \right\}  \geq\frac{3}{4}%
d_{0}.\text{ }%
\]
and
\[
\max\left\{  \left\vert \lambda\right\vert ,\ \lambda\in\sigma\left(  L_{\xi
}|_{R\left(  P_{\xi}^{0}\right)  }\right)  \right\}  \leq a\left(
\epsilon\right)  ,
\]
with $a\left(  \epsilon\right)  \rightarrow0$ when $\epsilon\rightarrow0$.
Denote
\[
E_{0}=\ker L_{\xi_{0}},\ E_{1}=\left(  \ker L_{\xi_{0}}\right)  ^{\perp
}=R\left(  L_{\xi_{0}}\right)  ,
\]
and
\[
E_{0}^{\xi}=R\left(  P_{\xi}^{0}\right)  ,\ E_{1}^{\xi}=R\left(  P_{\xi}%
^{1}\right)  .
\]
It is easy to show that $E_{1}^{\xi}$ can be written as a graph of a $O\left(
\epsilon\right)  $-bounded operator $S_{\xi}:E_{1}\rightarrow E_{0}$. That is,
let $\tilde{S}_{\xi}=I+S_{\xi}$, then $E_{1}^{\xi}=\tilde{S}_{\xi}\left(
E_{1}\right)  $. For any $u\in H^{1}(\mathbb{T}_{2\pi})$, let
\[
u=P_{\xi}^{0}u+P_{\xi}^{1}u=u^{0}+u^{1},
\]
then the equation $\partial_{t}u=L_{\xi}J_{\xi}u$ becomes
\begin{equation}
\partial_{t}u^{0}=P_{\xi}^{0}L_{\xi}J_{\xi}u^{0}+P_{\xi}^{0}L_{\xi}J_{\xi
}u^{1}, \label{eqn-u-0}%
\end{equation}%
\begin{equation}
\partial_{t}u^{1}=P_{\xi}^{1}L_{\xi}J_{\xi}u^{0}+P_{\xi}^{1}L_{\xi}J_{\xi
}u^{1}. \label{eqn-u-1}%
\end{equation}
We will show that: For any $\varepsilon>0\ $there exist $C(\varepsilon),$
$\epsilon>0\ $such that
\begin{equation}
\Vert e^{tP_{\xi}^{1}L_{\xi}J_{\xi}|_{E_{1}^{\xi}}}u\Vert_{H^{1}%
(\mathbb{T}_{2\pi})}\leqslant C(\varepsilon)e^{\left(  \lambda_{0}%
+\frac{\varepsilon}{2}\right)  t}\Vert u\Vert_{H^{1}(\mathbb{T}_{2\pi}%
)},\text{ }\forall t>0, \label{estimate-u-1-semigroup}%
\end{equation}
holds for $\xi\in(\xi_{0}-\epsilon,\xi_{0}+\epsilon)$. Assuming
(\ref{estimate-u-1-semigroup}), we now show (\ref{estimate-local-dual}) for
$\xi$ $\in(\xi_{0}-\epsilon,\xi_{0}+\epsilon)$. First, by (\ref{eqn-u-1}) we
have%
\begin{align}
&  \ \ \ \ \ \left\Vert u^{1}\left(  t\right)  \right\Vert _{H^{1}%
}\label{estimate-u-1}\\
&  \leq\left\Vert e^{tP_{\xi}^{1}L_{\xi}J_{\xi}|_{E_{1}^{\xi}}}u^{1}\left(
0\right)  \right\Vert _{H^{1}}+\left\Vert \int_{0}^{t}e^{\left(  t-s\right)
P_{\xi}^{1}L_{\xi}J_{\xi}|_{E_{1}^{\xi}}}P_{\xi}^{1}L_{\xi}J_{\xi}u^{0}\left(
s\right)  ds\right\Vert _{H^{1}}\nonumber\\
&  \leq C(\varepsilon)\left(  e^{\left(  \lambda_{0}+\frac{\varepsilon}%
{2}\right)  t}\Vert u^{1}\left(  0\right)  \Vert_{H^{1}}+\int_{0}%
^{t}e^{\left(  \lambda_{0}+\frac{\varepsilon}{2}\right)  \left(  t-s\right)
}\Vert u^{0}\left(  s\right)  \Vert_{H^{1}}ds\right)  .\nonumber
\end{align}
Since the operator $P_{\xi}^{0}L_{\xi}J_{\xi}$ is finite ranked and
$\left\Vert P_{\xi}^{0}L_{\xi}J_{\xi}\right\Vert _{H^{1}}\leq Ca\left(
\epsilon\right)  $ for some constant $C$, so from (\ref{eqn-u-0}) we have
\begin{equation}
\ \left\Vert u^{0}\left(  t\right)  \right\Vert _{H^{1}}\leq e^{Ca\left(
\epsilon\right)  t}\left\Vert u^{0}\left(  0\right)  \right\Vert _{H^{1}%
}+Ca\left(  \epsilon\right)  \int_{0}^{t}e^{Ca\left(  \epsilon\right)  \left(
t-s\right)  }\Vert u^{1}\left(  s\right)  \Vert_{H^{1}}ds.
\label{estimate-u-0}%
\end{equation}
We choose $\epsilon$ small enough such that $Ca\left(  \epsilon\right)
<\frac{\varepsilon}{2}$. Plugging above into (\ref{estimate-u-1}), we get
\begin{align*}
\left\Vert u^{1}\left(  t\right)  \right\Vert _{H^{1}}  &  \leq C^{\prime
}C(\varepsilon)e^{\left(  \lambda_{0}+\frac{\varepsilon}{2}\right)  t}\left(
\Vert u^{1}\left(  0\right)  \Vert_{H^{1}}+\left\Vert u^{0}\left(  0\right)
\right\Vert _{H^{1}}\right) \\
&  \ \ \ \ \ \ \ +Ca\left(  \epsilon\right)  C(\varepsilon)\int_{0}%
^{t}e^{\left(  \lambda_{0}+\frac{\varepsilon}{2}\right)  \left(  t-s\right)
}\int_{0}^{s}e^{\frac{\varepsilon}{2}\left(  s-\tau\right)  }\Vert
u^{0}\left(  \tau\right)  \Vert_{H^{1}}d\tau ds\\
&  \leq C^{\prime\prime}C(\varepsilon)e^{\left(  \lambda_{0}+\frac
{\varepsilon}{2}\right)  t}\Vert u\left(  0\right)  \Vert_{H^{1}}\\
\ \ \ \  &  \ \ \ \ \ \ \ +Ca\left(  \epsilon\right)  C(\varepsilon)e^{\left(
\lambda_{0}+\frac{\varepsilon}{2}\right)  t}\int_{0}^{t}e^{-\frac{\varepsilon
}{2}\tau}\Vert u^{0}\left(  \tau\right)  \Vert_{H^{1}}\int_{\tau}%
^{t}e^{-\lambda_{0}s}dsd\tau\\
&  \leq C^{\prime\prime}C(\varepsilon)e^{\left(  \lambda_{0}+\frac
{\varepsilon}{2}\right)  t}\Vert u\left(  0\right)  \Vert_{H^{1}}+Ca\left(
\epsilon\right)  C(\varepsilon)e^{\left(  \lambda_{0}+\frac{\varepsilon}%
{2}\right)  t}\int_{0}^{t}e^{-\left(  \lambda_{0}+\frac{\varepsilon}%
{2}\right)  \tau}\Vert u^{0}\left(  \tau\right)  \Vert_{H^{1}}d\tau,
\end{align*}
where $C^{\prime},C^{\prime\prime}$ are some constants independent of
$\epsilon$. Define
\[
y\left(  t\right)  =e^{-\left(  \lambda_{0}+\frac{\varepsilon}{2}\right)
t}\Vert u^{1}\left(  t\right)  \Vert_{H^{1}},
\]
then above inequality becomes
\[
y\left(  t\right)  \leq C^{\prime\prime}C(\varepsilon)\Vert u\left(  0\right)
\Vert_{H^{1}}+Ca\left(  \epsilon\right)  C(\varepsilon)\int_{0}^{t}y\left(
\tau\right)  d\tau.
\]
Choose $\epsilon$ further small such that $Ca\left(  \epsilon\right)
C(\varepsilon)<\frac{\varepsilon}{2}$. Then by Gronwall's inequality, we have
\[
y\left(  t\right)  \lesssim C(\varepsilon)e^{\frac{\varepsilon}{2}t}\Vert
u\left(  0\right)  \Vert_{H^{1}},
\]
that is,
\[
\left\Vert u^{1}\left(  t\right)  \right\Vert _{H^{1}}\lesssim C(\varepsilon
)e^{\left(  \lambda_{0}+\varepsilon\right)  t}\Vert u\left(  0\right)
\Vert_{H^{1}}.
\]
Plugging above estimate into (\ref{estimate-u-0}), we also get
\[
\left\Vert u^{0}\left(  t\right)  \right\Vert _{H^{1}}\lesssim C(\varepsilon
)e^{\left(  \lambda_{0}+\varepsilon\right)  t}\Vert u\left(  0\right)
\Vert_{H^{1}}.
\]
Combining above, we have
\[
\left\Vert u\left(  t\right)  \right\Vert _{H^{1}}\lesssim C(\varepsilon
)e^{\left(  \lambda_{0}+\varepsilon\right)  t}\Vert u\left(  0\right)
\Vert_{H^{1}},
\]
and thus (\ref{estimate-local-dual}) is proved. It remains to prove
(\ref{estimate-u-1-semigroup}). Since%
\[
P_{\xi}^{1}L_{\xi}J_{\xi}|_{E_{1}^{\xi}}=L_{\xi}|_{E_{1}^{\xi}}P_{\xi}%
^{1}J_{\xi}L_{\xi}|_{E_{1}^{\xi}}\left(  L_{\xi}|_{E_{1}^{\xi}}\right)  ^{-1}%
\]
and
\[
\left\Vert \left(  L_{\xi}|_{E_{1}^{\xi}}\right)  ^{-1}\right\Vert
_{H^{1}\rightarrow H^{1+m}}\lesssim\frac{1}{d_{0}},
\]
to prove (\ref{estimate-u-1-semigroup}) it suffices to show that there exist
$C(\varepsilon),$ $\epsilon>0\ $such that
\begin{equation}
\left\Vert e^{tP_{\xi}^{1}J_{\xi}L_{\xi}|_{E_{1}^{\xi}}}\right\Vert _{H^{1+m}%
}\leqslant C(\varepsilon)e^{\left(  \lambda_{0}+\frac{\varepsilon}{2}\right)
t},\text{ }\forall t>0, \label{estimate-E-1-semigroup}%
\end{equation}
for $\xi\in(\xi_{0}-\epsilon,\xi_{0}+\epsilon)$. Again, it is enough to
estimate $e^{tP_{\xi}^{1}J_{\xi}L_{\xi}|_{E_{1}^{\xi}}}$ on the energy space
$H^{\frac{m}{2}}$ and then apply $P_{\xi}^{1}J_{\xi}L_{\xi}|_{E_{1}^{\xi}}$
repeatedly (and by interpolation) to get the estimates for $s>\frac{m}{2}$. We
will study the semigroup generated by $P_{\xi}^{1}J_{\xi}L_{\xi}|_{E_{1}^{\xi
}}$ on $H^{\frac{m}{2}}\ $via the perturbation of the semigroup generated by
$P_{\xi_{0}}^{1}J_{\xi_{0}}L_{\xi_{0}}|_{E_{1}}$. First, we use the transform
$\tilde{S}_{\xi}:E_{1}\rightarrow E_{1}^{\xi}$ to study the conjugated
operators on the same space $E_{1}$. Notice that $\left(  \tilde{S}_{\xi
}\right)  ^{-1}:E_{1}^{\xi}\rightarrow E_{1}$ is exactly the projection
operator $P_{\xi_{0}}^{1}$. Therefore the $\tilde{S}_{\xi}-$conjugated
operator can be written in a Hamiltonian form
\[
\tilde{S}_{\xi}^{-1}P_{\xi}^{1}J_{\xi}L_{\xi}|_{E_{1}^{\xi}}\tilde{S}_{\xi
}=P_{\xi_{0}}^{1}P_{\xi}^{1}J_{\xi}\left(  P_{\xi}^{1}\right)  ^{\ast}\left(
P_{\xi_{0}}^{1}\right)  ^{\ast}\left(  \tilde{S}_{\xi}\right)  ^{\ast}\left(
P_{\xi}^{1}\right)  ^{\ast}L_{\xi}P_{\xi}^{1}\tilde{S}_{\xi}=\tilde{J}_{\xi
}\tilde{L}_{\xi},
\]
where
\[
\tilde{J}_{\xi}=P_{\xi_{0}}^{1}P_{\xi}^{1}J_{\xi}\left(  P_{\xi}^{1}\right)
^{\ast}\left(  P_{\xi_{0}}^{1}\right)  ^{\ast}:\left(  E_{1}\right)  ^{\ast
}\rightarrow E_{1}%
\]
and
\[
\tilde{L}_{\xi}=\left(  \tilde{S}_{\xi}\right)  ^{\ast}\left(  P_{\xi}%
^{1}\right)  ^{\ast}L_{\xi}P_{\xi}^{1}\tilde{S}_{\xi}:E_{1}\rightarrow\left(
E_{1}\right)  ^{\ast}%
\]
are anti-selfadjoint and self-adjoint respectively. We also write
\[
P_{\xi_{0}}^{1}J_{\xi_{0}}L_{\xi_{0}}|_{E_{1}}=P_{\xi_{0}}^{1}J_{\xi_{0}%
}\left(  P_{\xi_{0}}^{1}\right)  ^{\ast}\left(  P_{\xi_{0}}^{1}\right)
^{\ast}L_{\xi_{0}}P_{\xi_{0}}^{1}=\tilde{J}_{\xi_{0}}\tilde{L}_{\xi_{0}},
\]
where%
\[
\tilde{J}_{\xi_{0}}=P_{\xi_{0}}^{1}J_{\xi_{0}}\left(  P_{\xi_{0}}^{1}\right)
^{\ast},\ \tilde{L}_{\xi_{0}}=\left(  P_{\xi_{0}}^{1}\right)  ^{\ast}%
L_{\xi_{0}}P_{\xi_{0}}^{1}.
\]
We note that the spectrum of $\tilde{J}_{\xi}\tilde{L}_{\xi}$ is discrete,
$n^{-}\left(  \tilde{L}_{\xi}\right)  \leq n^{-}\left(  L_{\xi}\right)
<\infty$. Moreover, the maximal growth rate of the eigenvalues of $P_{\xi_{0}%
}^{1}J_{\xi_{0}}L_{\xi_{0}}$%
$\vert$%
$_{E^{1}}$ is still $\lambda_{0}$. Therefore by the similar proof as in Lemma
\ref{lemma-semigroup-whitham-localized} or Lemma 11.2 in
\cite{lin-zeng-Hamiltonian}, to prove the estimate
(\ref{estimate-E-1-semigroup}) in $H^{\frac{m}{2}}$, it suffices to show that
the resolvent $(\lambda-\tilde{J}_{\xi}\tilde{L}_{\xi})^{-1}\ $is continuous
for $\xi$ near $\xi_{0}$. We have
\[
\tilde{J}_{\xi}\tilde{L}_{\xi}-\tilde{J}_{\xi_{0}}\tilde{L}_{\xi_{0}}=\left(
\tilde{J}_{\xi}-\tilde{J}_{\xi_{0}}\right)  \tilde{L}_{\xi_{0}}+\tilde{J}%
_{\xi}\left(  \tilde{L}_{\xi}-\tilde{L}_{\xi_{0}}\right)  .
\]
In the above,
\begin{align*}
\tilde{J}_{\xi}-\tilde{J}_{\xi_{0}}  &  =P_{\xi_{0}}^{1}\left(  1-P_{\xi}%
^{0}\right)  J_{\xi}\left(  P_{\xi_{0}}^{1}\left(  1-P_{\xi}^{0}\right)
\right)  ^{\ast}-P_{\xi_{0}}^{1}J_{\xi_{0}}\left(  P_{\xi_{0}}^{1}\right)
^{\ast}\\
&  =-P_{\xi_{0}}^{1}P_{\xi}^{0}P_{\xi}^{0}J_{\xi}\left(  P_{\xi_{0}}^{1}%
P_{\xi}^{1}\right)  ^{\ast}-\left(  P_{\xi_{0}}^{1}P_{\xi}^{0}P_{\xi}%
^{0}J_{\xi}\left(  P_{\xi_{0}}^{1}P_{\xi}^{1}\right)  ^{\ast}\right)  ^{\ast
}\\
&  \ \ \ \ \ \ \ -P_{\xi_{0}}^{1}\left(  J_{\xi}-J_{\xi_{0}}\right)  \left(
P_{\xi_{0}}^{1}\right)  ^{\ast}\\
&  =O\left(  \left\vert \xi-\xi_{0}\right\vert \right)  ,
\end{align*}
since $J_{\xi}-J_{\xi_{0}}=O\left(  \left\vert \xi-\xi_{0}\right\vert \right)
,$
\[
P_{\xi_{0}}^{1}P_{\xi}^{0}=O\left(  \left\vert \xi-\xi_{0}\right\vert \right)
,\ \ P_{\xi}^{0}J_{\xi}=O\left(  1\right)  ,\ P_{\xi_{0}}^{1}P_{\xi}%
^{1}=O\left(  1\right)  .
\]
Also,
\begin{align*}
\tilde{L}_{\xi}-\tilde{L}_{\xi_{0}}  &  =\left(  \tilde{S}_{\xi}\right)
^{\ast}\left(  P_{\xi}^{1}\right)  ^{\ast}L_{\xi}P_{\xi}^{1}\tilde{S}_{\xi
}-\left(  P_{\xi_{0}}^{1}\right)  ^{\ast}L_{\xi_{0}}P_{\xi_{0}}^{1}\\
&  =\left(  P_{\xi_{0}}^{1}\right)  ^{\ast}\left(  L_{\xi}-L_{\xi_{0}}\right)
P_{\xi_{0}}^{1}+\left(  P_{\xi}^{1}\tilde{S}_{\xi}-P_{\xi_{0}}^{1}\right)
^{\ast}L_{\xi}P_{\xi}^{1}\tilde{S}_{\xi}\\
\ \ \ \ \  &  \ \ \ \ \ \ +\left(  P_{\xi_{0}}^{1}\right)  ^{\ast}L_{\xi
}\left(  P_{\xi}^{1}\tilde{S}_{\xi}-P_{\xi_{0}}^{1}\right)  ,
\end{align*}
where
\[
P_{\xi}^{1}\tilde{S}_{\xi}-P_{\xi_{0}}^{1}=P_{\xi}^{1}-P_{\xi_{0}}^{1}+P_{\xi
}^{1}S_{\xi}=O\left(  \left\vert \xi-\xi_{0}\right\vert \right)  .
\]
Thus by similar arguments as in the proof of Lemma
\ref{lemma-semigroup-whitham-localized} or Lemma 11.2 in
\cite{lin-zeng-Hamiltonian}, we can show the continuity of the resolvent
$(\lambda-\tilde{J}_{\xi}\tilde{L}_{\xi})^{-1}\ $for $\xi$ near $\xi_{0}$.
This finishes the proof of the Lemma.
\end{proof}

\section{Nonlinear Modulational Instability (multi-periodic)}

In this section, we prove that linearly modulationally unstable traveling
waves are nonlinearly orbitally unstable under multi-periodic perturbations.
First, by the definition (\ref{defn-MI}) of linear modulational instability
and the remark thereafter, there exists an interval $I_{0}\subset\left[
0,1\right]  $ such that for any $k\in I_{0}$, there exists an unstable
solution $e^{\lambda\left(  k\right)  t}e^{ikx}v_{k}\left(  x\right)  $ with
$\operatorname{Re}\lambda\left(  k\right)  >0$ and $2\pi-$periodic
$v_{k}\left(  x\right)  $ to the linearized equation (\ref{kdvln}). So we can
pick an rational number $k_{0}=\frac{p}{q}\in I_{0}$ with $p,q\in\mathbb{N}$.
Then $e^{ik_{0}x}v_{k_{0}}\left(  x\right)  $ is a $2\pi q$-periodic unstable
eigenfunction to the operator $JL$ in $L^{2}{(\mathbb{T}_{2\pi q})}$. It leads
us to consider the nonlinear instability of $u_{c}$ in ${L^{2}(\mathbb{T}%
_{2\pi q})}$.

The proof of Theorem \ref{thm-smooth-f} i) uses the strategy in
\cite{grenier-2000}, by constructing higher order approximation solutions and
then using the energy estimates to overcome the loss of derivative.

The following energy estimate will be used in the proof later. We use
$\mathbb{T}$ for ${\mathbb{T}_{2\pi q}}$ below.

\begin{lemma}
\label{lemma-energy-estimate}Consider the solution of the following equation%
\begin{equation}
\partial_{t}v-c\partial_{x}v+\partial_{x}\mathcal{M}v+\partial_{x}%
(f(u_{c}+U+v)-f(u_{c}+U))=R,\label{eqn-energy-estimate}%
\end{equation}%
\[
v(0,\cdot)=0,
\]
where $U\left(  t,\cdot\right)  \in H^{4}(\mathbb{T})$ and $R\left(
t,\cdot\right)  \in H^{2}(\mathbb{T})$ are given and $f\in C^{\infty
}(\mathbf{R})$.$\ $Assume that
\[
\sup_{0\leq t\leq T}\left\Vert U\right\Vert \left(  t\right)  _{H^{4}%
(\mathbb{T})}+\left\Vert v\right\Vert _{H^{2}(\mathbb{T})}\left(  t\right)
\leq\beta\text{,}%
\]
then there exists a constant $C\left(  \beta\right)  $ such that for $0\leq
t\leq T,$
\begin{equation}
\partial_{t}\left\Vert v\right\Vert _{H^{2}(\mathbb{T})}\leq C\left(
\beta\right)  \left\Vert v\right\Vert _{H^{2}(\mathbb{T})}+\left\Vert
R\right\Vert _{H^{2}(\mathbb{T})}.\label{estimate-H2-energy}%
\end{equation}

\end{lemma}

\begin{proof}
We write
\[
f(u_{c}+U+v)-f(u_{c}+U)=\int_{0}^{1}f^{\prime}\left(  u_{c}+U+\tau v\right)
d\tau v.
\]
First, taking the inner product of (\ref{eqn-energy-estimate}) with $v$ and
integrating by parts, we have
\begin{align*}
\frac{1}{2}\partial_{t}\left\Vert v\right\Vert _{L^{2}(\mathbb{T})}^{2} &
=-\left(  \left(  \int_{0}^{1}f^{\prime}\left(  u_{c}+U+\tau v\right)  d\tau
v\right)  _{x},v\right)  +\left(  R,v\right)  \\
&  =-\frac{1}{2}\int_{\mathbb{T}}\left(  \int_{0}^{1}f^{\prime}\left(
u_{c}+U+\tau v\right)  d\tau\right)  _{x}v^{2}dx+\left(  R,v\right)  \\
&  \leq C\left(  \beta\right)  \left\Vert f\left(  s\right)  \right\Vert
_{C^{2}\left(  \left\vert s\right\vert \leq\Vert u_{c}\Vert_{\infty}%
+C\beta\right)  }\left\Vert v\right\Vert _{L^{2}(\mathbb{T})}^{2}+\left\Vert
R\right\Vert _{L^{2}(\mathbb{T})}\left\Vert v\right\Vert _{L^{2}(\mathbb{T})},
\end{align*}
where in the above we use the fact that $\partial_{x}\mathcal{M}$ is
anti-selfadjoint and
\[
\left\Vert v\right\Vert _{\infty}+\left\Vert \partial_{x}v\right\Vert
_{\infty}\leq C\left\Vert v\right\Vert _{H^{2}(\mathbb{T})}.
\]
Thus
\begin{equation}
\partial_{t}\left\Vert v\right\Vert _{L^{2}(\mathbb{T})}\leq C\left(
\beta\right)  \left\Vert v\right\Vert _{L^{2}(\mathbb{T})}+\left\Vert
R\right\Vert _{L^{2}(\mathbb{T})}.\label{inequality-L2}%
\end{equation}
Next, applying $\partial_{x}^{2}$ to (\ref{eqn-energy-estimate}) and then
taking the inner product with $\partial_{x}^{2}v$, we get
\begin{equation}
\frac{1}{2}\partial_{t}\Vert\partial_{x}^{2}v\Vert_{L^{2}(\mathbb{T})}%
^{2}=-\left(  \left(  \int_{0}^{1}f^{\prime}\left(  u_{c}+U+\tau v\right)
d\tau v\right)  _{xxx},v_{xx}\right)  +\left(  R_{xx},v_{xx}\right)
.\label{eqn-integral-second-derivative}%
\end{equation}
By direct computation and integration by parts, we can show that for $0<t\leq
T$, there exists a constant $C(\beta)$, such that
\[
\left\vert \left(  \left(  \int_{0}^{1}f^{\prime}\left(  u_{c}+U+\tau
v\right)  d\tau v\right)  _{xxx},v_{xx}\right)  \right\vert \leq C(\beta)\Vert
v\Vert_{H^{2}(\mathbb{T})}^{2}.
\]
We only sketch the estimates of the terms involving $\partial_{x}^{3}v$. One
such term is

\ \ \ \ \
\begin{align*}
&  \ \ \ \ \ \left\vert \left(  \int_{0}^{1}f^{\prime}\left(  u_{c}+U+\tau
v\right)  d\tau\ v_{xxx},v_{xx}\right)  \right\vert \\
&  =\left\vert \int_{\mathbb{T}}f^{\prime}\left(  u_{c}+U+\tau v\right)
d\tau\frac{1}{2}\partial_{x}\left(  v_{xx}\right)  ^{2}dx\right\vert
\newline\\
&  =\frac{1}{2}\left\vert \int_{\mathbb{T}}\left(  \int_{0}^{1}f^{\prime
}\left(  u_{c}+U+\tau v\right)  d\tau\right)  _{x}\left(  v_{xx}\right)
^{2}dx\right\vert \\
&  \leq C\left(  \beta\right)  \left\Vert v_{xx}\right\Vert _{L^{2}%
(\mathbb{T})}^{2},
\end{align*}
$\newline$\newline and another term
\[
\left(  \int_{0}^{1}f^{(4)}\left(  u_{c}+U+\tau v\right)  \tau^{3}
d\tau\ vv_{xxx},v_{xx}\right)
\]
can be handled similarly. Thus by (\ref{eqn-integral-second-derivative}), we
have
\[
\partial_{t}\left\Vert v_{xx}\right\Vert _{L^{2}(\mathbb{T})}\leq C\left(
\beta\right)  \left\Vert v_{xx}\right\Vert _{L^{2}(\mathbb{T})}+\left\Vert
R_{xx}\right\Vert _{L^{2}(\mathbb{T})},
\]
and combined with (\ref{inequality-L2}) this proves (\ref{estimate-H2-energy}).
\end{proof}

Now we are ready to prove nonlinear modulational instability for
multi-periodic perturbations.

\begin{proof}
[Proof of Theorem \ref{thm-smooth-f} i)]Let $v_{g}\left(  x\right)  \ $be the
eigenfunction associated with the most unstable eigenvalue $\lambda\ $of
$JL\ $in $L^{2}\left(  \mathbb{T}\right)  $. By Lemmas
\ref{lem-eigenfunction-kdv} and \ref{lem-eigenfunction-whitham}, $v_{g}\in
H^{s}\left(  \mathbb{T}\right)  \ $for any $s\geq0$. We construct an
approximate solution $U^{app}\ $to (\ref{kdvtv}) of the form
\begin{equation}
U^{app}(t,x)=u_{c}(x)+\sum_{j=1}^{N}\delta^{j}U_{j}(t,x),\label{uapp}%
\end{equation}
where
\begin{equation}
U_{1}(t,x)=v_{g}(x)e^{\lambda t}+\bar{v}_{g}(x)e^{\bar{\lambda}t}%
,\label{u1peri}%
\end{equation}
is the most rapidly growing real-valued $2\pi q$-periodic solution of the
linearized equation (\ref{kdvln}). The integer $N$ is chosen such that
$\left(  N+1\right)  \operatorname{Re}\lambda>C\left(  1\right)  $, where the
constant $C\left(  1\right)  $ is the one in the energy estimate
(\ref{estimate-H2-energy}) with $\beta=1$.

Now we construct the terms $U_{2},\cdots,U_{N}$. By the Taylor expansion formula,%

\begin{align}
f(U^{app})-f(u_{c})= & \sum_{k=1}^{N}\frac{f^{(k)}(u_{c})}{k!}\left(
\sum_{j=1}^{N}\delta^{j}U_{j}\right)  ^{k}\label{identity-difference-f}\\
&  +\int_{0}^{1}\frac{f^{(N+1)}\left(  u_{c}+\tau\sum_{j=1}^{N}\delta^{j}%
U_{j}\right)  }{N!}\left(  1-\tau\right)  ^{N}d\tau\left(  \sum_{j=1}%
^{N}\delta^{j}U_{j}\right)  ^{N+1}.\nonumber
\end{align}
Since $u_{c}$ is a stationary solution to (\ref{kdvtv}) and $U_{1}$ satisfies
the linearized equation
\[
\partial_{t}U_{1}-c\partial_{x}U_{1}+\partial_{x}(\mathcal{M}U_{1}+f^{\prime
}(u_{c})U_{1})=0,
\]
by using (\ref{identity-difference-f}) we have
\begin{align*}
&  \partial_{t}U^{app}-c\partial_{x}U^{app}+\partial_{x}(\mathcal{M}%
U^{app}+f(U^{app}))\\
=  & \sum_{j=2}^{N}\delta^{j}\left(  \partial_{t}U_{j}-c\partial_{x}%
U_{j}+\partial_{x}(\mathcal{M}U_{j}+f^{\prime}(u_{c})U_{j})+\partial_{x}%
P_{j}(u_{c};U_{1},U_{2},\cdots,U_{j-1})\right) \\
&  +\sum_{j=N+1}^{N^{N}}\delta^{j}\partial_{x}Q_{j}(u_{c};U_{1},U_{2}%
,\cdots,U_{N})+\partial_{x}\left( g\left(  u_{c};U_{1},U_{2},\cdots
,U_{N}\right)  \left(  \sum_{j=1}^{N}\delta^{j}U_{j}\right)  ^{N+1}\right) ,
\end{align*}
where
\[
g\left(  u_{c};U_{1},U_{2},\cdots,U_{N}\right)  =\int_{0}^{1}\frac
{f^{(N+1)}\left(  u_{c}+\tau\sum_{j=1}^{N}\delta^{j}U_{j}\right)  }{N!}\left(
1-\tau\right)  ^{N}d\tau,
\]
and $P_{j},\ Q_{j}$ are polynomials of $U_{1},\cdots,U_{N}\ $with degree
$j\ $such that
\begin{align*}
&  \sum_{k=2}^{N}\frac{f^{(k)}(u_{c})}{k!}\left(  \sum_{j=1}^{N}\delta
^{j}U_{j}\right)  ^{k}\\
=  & \sum_{j=2}^{N}\delta^{j}P_{j}(u_{c};U_{1},U_{2},\cdots,U_{j-1}%
)+\sum_{j=N+1}^{N^{N}}\delta^{j}Q_{j}(u_{c};U_{1},U_{2},\cdots,U_{N}).
\end{align*}
For $j=2,\cdots,N$, we define $U_{j}$ be the solution of
\begin{equation}%
\begin{cases}
\partial_{t}U_{j}=JLU_{j}+\partial_{x}P_{j}(u_{c};U_{1},U_{2},\cdots
,U_{j-1}),\\
U_{j}(0,\cdot)=0,
\end{cases}
\label{equj}%
\end{equation}

Now we estimate $U_{j}$ for $j\geq2$. First, by Lemmas
\ref{lem-eigenfunction-kdv} and \ref{lem-eigenfunction-whitham}, one has
\begin{equation}
\Vert U_{1}\left(  t\right)  \Vert_{H^{l}(\mathbb{T})}\leq Ce^{(Re\lambda)t},
\label{U1}%
\end{equation}
where $l=s+N$. By (\ref{equj}), $U_{2}$ satisfies the equation
\begin{equation}
\partial_{t}U_{2}=JLU_{2}+\partial_{x}P_{2}(U_{1}),\ U_{2}\left(  0\right)
=0, \label{equ2}%
\end{equation}
where $P_{2}(U_{1})=\frac{1}{2}f^{\prime\prime}(u_{c})U_{1}^{2}$. By
(\ref{U1}), we have
\[
\Vert\partial_{x}P_{2}(U_{1})\Vert_{H^{l-1}(\mathbb{T})}\leqslant C\left(
l\right)  e^{2\operatorname{Re}\lambda t}.
\]
Then, it follows from Lemma \ref{semigestim} that
\[
\Vert U_{2}(t,x)\Vert_{H^{l-1}(\mathbb{T})}\leqslant C\left(  l\right)
e^{2\operatorname{Re}\lambda t}.
\]
By induction, for each $2<j\leqslant N$, we have
\[
\Vert\partial_{x}P_{j}(U_{1},\cdots,u_{j-1})\Vert_{H^{l+1-j}(\mathbb{T}%
)}\leqslant C\left(  j,l\right)  e^{j\operatorname{Re}\lambda t},
\]
and then by Lemma \ref{semigestim}
\[
\Vert U_{j}(t,x)\Vert_{H^{l+1-j}(\mathbb{T})}\leqslant C\left(  j,l\right)
e^{j\operatorname{Re}\lambda t}.
\]
Therefore, there exists a constant $C(N,s)$, such that
\begin{equation}
\Vert U_{j}(t,x)\Vert_{H^{l+1-j}(\mathbb{T})}\leqslant
C(N,s)e^{j\operatorname{Re}\lambda t},\text{ for }j=1,2,\cdots,N.
\label{estimate-U-j}%
\end{equation}

By the construction of $U^{app}$, we have
\begin{equation}
\partial_{t}U^{app}-c\partial_{x}U^{app}+\partial_{x}(\mathcal{M}%
U^{app}+f(U^{app}))=R_{app}, \label{EQUAPP}%
\end{equation}
where
\begin{equation}
R_{app}=\sum_{j=N+1}^{N^{N}}\delta^{j}\partial_{x}Q_{j}(u_{c};U_{1}%
,U_{2},\cdots,U_{N})+\partial_{x}\left( g\left(  u_{c};U_{1},U_{2}%
,\cdots,U_{N}\right)  \left(  \sum_{j=1}^{N}\delta^{j}U_{j}\right)
^{N+1}\right) . \label{Rapp}%
\end{equation}

Let $0<\theta<1$ to be determined and define $T^{\delta}$ by $\delta
e^{\operatorname{Re}\lambda T^{\delta}}=\theta$. Then $T^{\delta}=O\left(
\left\vert \ln\delta\right\vert \right)  $. Choose $s\geq4$ and recall that
$l-N=s$. Then by (\ref{estimate-U-j}), for any $N+1\leq j\leq N^{N}$, we have
\[
\Vert\partial_{x}Q_{j}(u_{c};U_{1},U_{2},\cdots,U_{N})\Vert_{H^{s}%
(\mathbb{T})}\leqslant C(N,s)e^{j\operatorname{Re}\lambda t}%
\]
and thus by (\ref{Rapp})
\begin{equation}
\left\Vert R_{app}\right\Vert _{H^{s}}\leq C\left(  N,s\right)  e^{\left(
N+1\right)  \operatorname{Re}\lambda t},\ \ \text{for }0\leq t\leq T^{\delta}.
\label{estimate-R-app}%
\end{equation}
Let $U_{\delta}(t,x)$ be the solution to (\ref{kdvtv}) with initial value
$u_{c}(x)+\delta U_{1}(0,x)$, and let $v=U_{\delta}-U^{app}$. Then by using
(\ref{EQUAPP}), one finds that $v$ satisfies the equation
\begin{equation}%
\begin{cases}
\partial_{t}v-c\partial_{x}v+\partial_{x}\mathcal{M}v+\partial_{x}%
(f(U^{app}+v)-f(U^{app}))=-R_{app}\\
v(0,\cdot)=0.
\end{cases}
\label{eqn-v-error}%
\end{equation}

Define $T_{1}$ to be the maximal time such that
\[
\left\Vert v\left(  t\right)  \right\Vert _{H^{2}}\leq\frac{1}{2},\ 0\leq
t\leq T_{1}.
\]
We claim that $T_{1}>T^{\delta}$ when $\theta$ is chosen to be small enough.
Suppose otherwise, $T_{1}\leq T^{\delta}$. Then for $0\leq t\leq T_{1}$, we
have
\begin{align*}
\left\Vert U^{app}-u_{c}\right\Vert _{H^{4}} &  \leq\sum_{j=1}^{N}\delta
^{j}\left\Vert U_{j}\right\Vert _{H^{s}}\leq C\left(  N,s\right)  \sum
_{j=1}^{N}\left(  \delta e^{\operatorname{Re}\lambda t}\right)  ^{j}\\
&  \leq\frac{C\theta}{1-\theta}\leq\frac{1}{2},
\end{align*}
when$\ \theta$ is small. Thus we have%
\[
\sup_{0\leq t\leq T_{1}}\left\Vert U^{app}-u_{c}\right\Vert _{H^{4}%
(\mathbb{T})}\left(  t\right)  +\left\Vert v\right\Vert _{H^{2}(\mathbb{T}%
)}\left(  t\right)  \leq1.
\]
By using Lemma \ref{lemma-energy-estimate} for the equation (\ref{eqn-v-error}%
), we have%

\begin{equation}
\partial_{t}\left\Vert v\right\Vert _{H^{2}}\leq C\left(  1\right)  \left\Vert
v\right\Vert _{H^{2}}+\left\Vert R_{app}\right\Vert _{H^{2}},\text{ for }0\leq
t\leq T_{1}. \label{Gronwall-v-H2}%
\end{equation}
Recall that $\left(  N+1\right)  \operatorname{Re}\lambda>C\left(  1\right)
$. So by using (\ref{estimate-R-app}) and the Gronwall's inequality, we obtain
from (\ref{Gronwall-v-H2}) that for $0\leq t\leq T_{1},$
\begin{equation}
\left\Vert v\right\Vert _{H^{2}}\left(  t\right)  \leq C\left(  N,s\right)
e^{\left(  N+1\right)  \operatorname{Re}\lambda t}\text{. }
\label{estimate-v-H2}%
\end{equation}
Thus
\[
\left\Vert v\right\Vert _{H^{2}}\left(  T_{1}\right)  \leq C\theta^{N+1}%
<\frac{1}{2},
\]
when $\theta$ is small. This is in contradiction to the definition of $T_{1}$
and the claim is proved. Moreover, for $0\leq t\leq T^{\delta}<T_{1}$, when
$\theta$ is small enough the estimate (\ref{estimate-v-H2}) is true by above
arguments. So there exist $C_{1},C_{2}>0$ such that%
\begin{align*}
&  \ \ \ \ \ \left\Vert U_{\delta}\left(  T^{\delta},x\right)  -u_{c}\left(
x\right)  \right\Vert _{L^{2}}\\
&  \geq\left\Vert U^{app}\left(  T^{\delta},x\right)  -u_{c}\left(  x\right)
\right\Vert _{L^{2}}-\left\Vert v\left(  T^{\delta},x\right)  \right\Vert
_{H^{2}}\\
&  \geq C_{1}\delta e^{\operatorname{Re}\lambda T^{\delta}}-C_{2}\left(
\delta e^{\operatorname{Re}\lambda T^{\delta}}\right)  ^{2}=C_{1}\theta
-C_{2}\theta^{2}\\
&  \geq\frac{1}{2}C_{1}\theta\text{,}%
\end{align*}
when $\theta\ $is small enough.

It remains to show that above nonlinear instability is also true in the
orbital distance. This can be done by using the argument in (\cite{gss90}). By
the previous estimates, there exists a constant $\widetilde{C}$, such that
\[
\Vert U_{\delta}(t,x)-u_{c}(x)\Vert_{H^{2}(\mathbb{T})}\leqslant
\widetilde{C}\theta,\quad\text{for }0<t\leqslant T^{\delta},
\]
where $\widetilde{C}\ $may depend on $\theta$, but is independent of $\delta$.
Denote
\[
V_{1}\left(  t,x\right)  =e^{-\operatorname{Re}\lambda t}U_{1}\left(
x,t\right)  =2\left(  \operatorname{Re}v_{g}\cos\left(  \operatorname{Im}%
\lambda t\right)  -\operatorname{Im}v_{g}\sin\left(  \operatorname{Im}\lambda
t\right)  \right)  ,
\]
then it is easy to see that for any $s\geq0$, there exist two constants
$c_{1}\left(  s\right)  ,c_{2}\left(  s\right)  >0\ $such that
\[
0<c_{1}\left(  s\right)  \leq\left\Vert V_{1}\right\Vert _{H^{s}}\leq
c_{2}\left(  s\right)  .
\]
Let $V_{1}^{\bot}(t,x)\ $be the projection of $V_{1}(t,x)\ $into $Z^{\bot}%
\ $in the $L^{2}\ $inner product, where
\[
Z^{\bot}=\{v\in L^{2}(\mathbb{T}):\langle v,\partial_{x}u_{c}\rangle=0\}.
\]
\newline Let $h(t)\ $be such that
\[
\Vert U_{\delta}(t,x)-u_{c}(x+h(t))\Vert_{L^{2}(\mathbb{T})}=\inf
_{y\in\mathbb{T}}\ \Vert U_{\delta}(t,x)-u_{c}(x+y)\Vert_{L^{2}(\mathbb{T})}.
\]
Then for $0<t\leqslant T^{\delta}$, we have%
\begin{align*}
&  \ \ \ \ \ \Vert u_{c}(x)-u_{c}(x+h(t))\Vert_{L^{2}(\mathbb{T})}\\
&  \leqslant\Vert U_{\delta}(t,x)-u_{c}(x)\Vert_{L^{2}(\mathbb{T})}+\Vert
U_{\delta}(t,x)-u_{c}(x+h(t))\Vert_{L^{2}(\mathbb{T})}\\
&  \leq2\Vert U_{\delta}(t,x)-u_{c}(x)\Vert_{L^{2}(\mathbb{T})}\leq2\tilde
{C}\theta,
\end{align*}
which implies $|h(t)|=O(\theta)$. So we can write%

\[
u_{c}(x+h)=u_{c}(x)+h\partial_{x}u_{c}(x)+O(\theta^{2}).
\]
This implies that%
\begin{align*}
&  \ \ \ \ \ \ |\langle U_{\delta}(x)-u_{c}(x+h(T^{\delta})),V_{1}^{\bot
}(T^{\delta},x)\rangle|\\
&  \geqslant|\langle U_{\delta}(x)-u_{c}(x),V_{1}^{\bot}(T^{\delta}%
,x)\rangle|-O(\theta^{2})\geq c_{0}\theta,
\end{align*}
for some $c_{0}>0$, when $\theta\ $is small enough. On the other hand, we have%
\begin{align*}
&  \;\ \ \ |\langle U_{\delta}(T^{\delta},x)-u_{c}(x+h(T^{\delta}%
)),V_{1}^{\bot}(T^{\delta},x)\rangle|\\
&  \leqslant\inf_{y\in\mathbb{T}}\Vert U(T^{\delta},x)-u_{c}(x+y)\Vert
_{L^{2}(\mathbb{T})}\Vert V_{1}^{\bot}(T^{\delta},x)\Vert_{L^{2}(\mathbb{T})},
\end{align*}
which implies that
\[
\inf_{y\in\mathbb{T}}\Vert U(T^{\delta},x)-u_{c}(x+y)\Vert_{L^{2}(\mathbb{T}%
)}\geq C^{\prime}\theta,
\]
for some $C^{\prime}>0$. This finishes the proof of Theorem \ref{thm-smooth-f} i).
\end{proof}

\section{Localized Nonlinear Modulational Instability}

\label{localized}

In this section, we prove nonlinear instability for localized perturbations.
Since the linearized operator $JL\ $(defined in (\ref{definition-J-L})) does
not have an unstable eigenvalue in $H^{s}\left(  \mathbf{R}\right)  $, we will
construct unstable initial data in the form of a wave package of unstable
eigenfunctions of $J_{k}L_{k}\ $where $k\ $is near the most unstable frequency
$k_{0}$. Without loss of generality, we can assume that $k_{0}\in\left[
0,\frac{1}{2}\right]  $. Indeed, if $k\in\left[  0,1\right]  \ $is an unstable
frequency in the sense that $J_{k}L_{k}\ $has an unstable eigenvalue, then
$-k,1-k\ $are also unstable frequencies. So we can always pick $k_{0}%
\in\left[  0,\frac{1}{2}\right]  \ $such that $J_{k_{0}}L_{k_{0}}\ $has the
most unstable eigenvalue $\lambda\left(  k_{0}\right)  $. More precisely, for
any $k\in\left[  0,1\right]  $, if $J_{k}L_{k}\ $has an unstable eigenvalue
$\lambda\ $then $\operatorname{Re}\lambda\leq\operatorname{Re}\lambda\left(
k_{0}\right)  $. To construct the unstable wave package, we choose a small
interval $I\subset\left[  0,\frac{1}{2}\right]  \ $and $I\ $is near $k_{0}$.
If $\left\vert I\right\vert \ $is small enough, then any $k\in I\ $is still an
unstable frequency since $J_{k}L_{k}\ $depends on $k\ $smoothly. In the case
when $\lambda_{k_{0}}\ $is a simple eigenvalue of $J_{k_{0}}L_{k_{0}}$, then
by the analytic perturbation theory (\cite{kato-book}) of linear operators,
there is a smooth curve of unstable eigenvalue $\lambda\left(  k\right)  \ $of
$J_{k}L_{k}$, with $k\in I$. Since $\operatorname{Re}\lambda(k)\ $is smooth in
the vicinity of $k_{0}$, and $\operatorname{Re}\lambda\left(  k\right)
\ $obtains its maximum at $k_{0}$, there exists an even number $l\geqslant2$,
such that%

\begin{equation}
\left[  Re(\lambda)\right]  ^{\prime}(k_{0})=\cdots=\left[  Re(\lambda
)\right]  ^{(l-1)}(k_{0})=0,\quad\left[  Re(\lambda)\right]  ^{(l)}(k_{0})<0.
\label{taylor-simple}%
\end{equation}

Now consider the general case when $\lambda_{k_{0}}$ is a multiple eigenvalue
of $J_{k_{0}}L_{k_{0}}$. Since the eigenvalues of $J_{k}L_{k}$ are all
discrete, we can use the analytic perturbation theory (\cite{kato-book}) of
eigenvalues of matrices to study the eigenvalues of $J_{k}L_{k}$ near $k_{0}$.
In this case, the eigenvalues of $J_{k}L_{k}$ near $k_{0}$ can be grouped in
the manner
\[
\left\{  \lambda_{1}\left(  k\right)  ,\cdots,\lambda_{p_{1}}\left(  k\right)
\right\}  ,\left\{  \lambda_{p_{1}+1}\left(  k\right)  ,\cdots,\lambda
_{p_{1}+p_{2}}\left(  k\right)  \right\}  ,\cdots
\]
such that each group constitutes a branch of an analytic function (defined
near $k_{0}$) with a branch point (if $p_{i}\geq2$) at $k=k_{0}$. Assume
$p_{1}\geq2$, then we have the following Puiseux series (see p. 65 of
\cite{kato-book}) for the first group $\left\{  \lambda_{1}\left(  k\right)
,\cdots,\lambda_{p_{1}}\left(  k\right)  \right\}  $
\begin{equation}
\lambda_{h+1}\left(  k\right)  =\lambda\left(  k_{0}\right)  +m_{1}\omega
^{h}\left(  k-k_{0}\right)  ^{1/p_{1}}+m_{2}\omega^{2h}\left(  k-k_{0}\right)
^{2/p_{1}}+\cdots,\label{expansion-Puiseux}%
\end{equation}
where $\omega=\exp\left(  2\pi i/p_{1}\right)  $ and $h=0,1,\cdots,p_{1}-1$.
In the next lemma, we show that the leading order term of $\lambda
_{h+1}\left(  k\right)  $ in (\ref{expansion-Puiseux}) is still given by
$\left(  k-k_{0}\right)  ^{l}$ for an even integer $l$.

\begin{lemma}
\label{lemma-Puisex}Let $p_{1}\geq2$, consider the Puiseux series
(\ref{expansion-Puiseux}) near $k_{0}$. If
\begin{equation}
\max\operatorname{Re}\lambda_{h+1}\left(  k\right)  \leq\operatorname{Re}%
\lambda\left(  k_{0}\right)  ,\ h=0,1,\cdots,p_{1}-1, \label{condition-max}%
\end{equation}
for $k$ in a neighborhood of $k_{0}$, then there exists an even integer $l$
such that
\[
\operatorname{Re}m_{1}=\cdots=\operatorname{Re}m_{lp_{1}-1}%
=0,\ \operatorname{Re}m_{lp_{1}}<0.
\]

\end{lemma}

\begin{proof}
Let $m_{n}$ be the first coefficient in (\ref{expansion-Puiseux}) such that
$\operatorname{Re}m_{n}\neq0$. Then by (\ref{condition-max}), we have
\[
\operatorname{Re}m_{n}\omega^{nh}\left(  k-k_{0}\right)  ^{n/p_{1}}%
\leq0,\ h=0,1,\cdots,p_{1}-1.
\]
This implies that:
\[
\operatorname{Re}m_{n}\exp\left(  \frac{2\pi inh}{p_{1}}\right)
\leq0,\ \ \text{when\ }k-k_{0}>0,
\]
and
\[
\operatorname{Re}m_{n}\exp\left(  \frac{\pi in\left(  2h+1\right)  }{p_{1}%
}\right)  \leq0,\ \text{when\ }k-k_{0}<0,
\]
for $h=0,1,\cdots,p_{1}-1$. So
\begin{equation}
\operatorname{Re}m_{n}\exp\left(  \frac{\pi ni}{p_{1}}j\right)  \leq0,\ 0\leq
j\leq2p_{1}-1. \label{condiiton-points}%
\end{equation}
If $n/p_{1}$ is not an integer, then we must have $m_{n}=0$. Since otherwise
if $m_{n}\neq0$, it is clearly impossible for all the $2p_{1}$ points
\[
m_{n}\exp\left(  \frac{\pi ni}{p_{1}}j\right)  ,\ \ 0\leq j\leq2p_{1}-1\
\]
to stay in the left half complex plane when $n/p_{1}$ is not an integer. If
$n/p_{1}$ is odd, then for (\ref{condiiton-points}) to hold true we must have
$\operatorname{Re}m_{n}=0$. So for $\operatorname{Re}m_{n}\neq0$, we must have
$n/p_{1}=l$ to be even. In this case, (\ref{condiiton-points}) implies that
$\operatorname{Re}m_{lp_{1}}<0.$
\end{proof}

Let $I\subset\left[  0,\frac{1}{2}\right]  $ be a small interval with $k_{0}$
being its right end point. Let $\lambda\left(  k\right)  ,\ k\in I$ be a curve
of unstable eigenvalues of $J_{k}L_{k}$ ending on the right at $\lambda\left(
k_{0}\right)  $, as determined by one of the functions in
(\ref{expansion-Puiseux}) when $\lambda\left(  k_{0}\right)  $ is a multiple
eigenvalue. Then by (\ref{taylor-simple}) when $\lambda\left(  k_{0}\right)  $
is simple or by Lemma \ref{lemma-Puisex} when $\lambda\left(  k_{0}\right)  $
is multiple, we have
\begin{equation}
\operatorname{Re}\lambda\left(  k\right)  -\operatorname{Re}\lambda\left(
k_{0}\right)  =-a_{0}\left(  k-k_{0}\right)  ^{l}+o\left(  \left(
k-k_{0}\right)  ^{l}\right)  , \label{taylor-real-lambda}%
\end{equation}
where $a_{0}<0$ and $l$ is even. Let $v_{1}(k,x)$ be the corresponding
eigenfunction of $\lambda\left(  k\right)  $ for $J_{k}L_{k}$, which depends
on $k$ continuously. By Lemmas \ref{lem-eigenfunction-kdv} and
\ref{lem-eigenfunction-whitham}, $v_{1}(k,x)\in H_{x}^{s}\left(
\mathbb{T}\right)  $ for any $s\geq0$ when $f$ is smooth.

Define the following wave packet consisting of unstable eigenfunctions with
frequencies in $I$,
\begin{equation}
u_{1}\left(  x\right)  =\int_{I}v_{1}(k,x)e^{ikx}dk+\int_{I}\overline
{v_{1}(k,x)}e^{-ikx}dk=2\operatorname{Re}\int_{I}v_{1}(k,x)e^{ikx}dk.
\label{initial data-line}%
\end{equation}

Since $I\cup-I\subset\left[  -\frac{1}{2},\frac{1}{2}\right]  $, so by Lemma
\ref{ptor},
\[
\left\Vert u_{1}\left(  x\right)  \right\Vert _{H^{s}\left(  R\right)  }%
^{2}\lesssim\int_{I}\left\Vert v_{1}(k,x)\right\Vert _{H_{x}^{s}(\mathbb{T}%
)}^{2}dk<\infty\text{. }%
\]

We will choose initial data $U_{\delta}\left(  0\right)  =u_{c}+\delta u_{1}$
to show nonlinear localized instability. First, we follow the arguments in
Section 8.5 of \cite{pazy} to prove the well-posedness of (\ref{kdvtv}) in the
space $u_{c}+H^{s}\left(  R\right)  $. The arguments can be also found in
\cite{kato-lecture} \cite{kato-quasilinear}.

\begin{lemma}
[Well Posedness]\label{lemma-wp} Assuming that $\mathcal{M}\in L\big(H^{\beta
}(\mathbf{R}),L^{2}(\mathbf{R})\big)$ ($\beta$ may be negative ) and $f\in
C^{s+2}\left(  \mathbf{R}\right)  $, where $s\geqslant\max\{1+\beta,1\}$ is an
even integer. Then for every $u_{0}\in B^{s}(\mathbf{R}):=\big\{u_{c}+w:w\in
H^{s}(\mathbf{R})\big\}$, there exists $T>0$, such that the Cauchy problem
\[%
\begin{cases}
\partial_{t}u-c\partial_{x}u+\partial_{x}(\mathcal{M}u+f(u))=0,\quad
\quad(t,x)\in\lbrack0,\infty)\times\mathbf{R}\\
u(0,x)=u_{0}(x)
\end{cases}
\]
has a unique solution $u\in C([0,T],B^{s}(\mathbf{R}))\cap C^{1}%
([0,T],B^{0}(\mathbf{R})).$
\end{lemma}

\begin{proof}
It is equivalent to prove that the following problem
\begin{equation}%
\begin{cases}
\partial_{t}w-c\partial_{x}w+\partial_{x}(\mathcal{M}w+f(u_{c}+w)-f(u_{c}%
))=0\\
w(0,x)=w_{0}%
\end{cases}
\label{w}%
\end{equation}
has an unique solution $w\in C([0,T],H^{s}(\mathbf{R}))\cap C^{1}%
([0,T],L^{2}(\mathbf{R})).$\newline

Rewrite the equation (\ref{w}) as
\[
\partial_{t}w+\partial_{x}(\mathcal{M}-c)w+f^{\prime}(u_{c}+w)\partial
_{x}w+\partial_{x}u_{c}\int_{0}^{1}f^{\prime\prime}(u_{c}+\tau w)wd\tau=0.
\]
Let $A_{0}=-c\partial_{x}+\partial_{x}\mathcal{M}.$ It is clear that
$D(A_{0})=H^{\sigma}(\mathbf{R})$, where $\sigma=\max\{1+\beta,1\}$.

For any $v\in H^{s}$ with $s\geq\sigma$, define $A_{1}(v):H^{1}(\mathbf{R}%
)\rightarrow L^{2}(\mathbf{R})$ as
\[
A_{1}(v)w=f^{\prime}(u_{c}+v)\partial_{x}w+\partial_{x}u_{c}\int_{0}%
^{1}f^{\prime\prime}(u_{c}+\tau v)wd\tau.
\]

Following the arguments in Section 8.5 of \cite{pazy}, we consider the
equation
\[
\partial_{t}w+A(v)w=0,
\]
where $A(v)=A_{0}+A_{1}(v)$. \newline

Let $B_{r}$ be the ball of radius $r>0$ in $H^{\sigma}(\mathbf{R})$. According
to Theorem 6.4.6 and Section 8.5 in \cite{pazy}, the following four conditions
guarantee the well-posedness of (\ref{w}):

(C1) There exists a constant $k$, such that if $\Vert w_{0}\Vert
_{H^{s}(\mathbf{R})}\leq r$, then
\[
\Vert A(v)w_{0}\Vert_{L^{2}(\mathbf{R})}\leqslant k,
\]
for every $v\in B_{r}$;

(C2) The family $A(v)$, $v\in B_{r}$ is a stable family in $L^{2}(\mathbf{R})$
(see Definition 6.4.1 in P. 200 of \cite{pazy});

(C3) There is an isomorphism of $H^{s}(\mathbf{R})$ onto $L^{2}(\mathbf{R})$
such that for every $v\in B_{r}$, $SA(v)S^{-1}-A(v)$ is a bounded operator in
$L^{2}(\mathbf{R})$ and $\Vert SA(v)S^{-1}-A(v)\Vert\leqslant C_{1}$;

(C4) For each $v\in B_{r}$, $D(A(v))\supset H^{s}(\mathbf{R})$, $A(v)$ is a
bounded linear operator from $H^{s}(\mathbf{R})$ into $L^{2}(\mathbf{R})$ and
\[
\Vert A(v_{1})-A(v_{2})\Vert_{L(H^{s}(\mathbf{R}),L^{2}(\mathbf{R}))}\leq
C_{1}\Vert v_{1}-v_{2}\Vert_{L^{2}(\mathbf{R})}.
\]

Since $\Vert w_{0}\Vert_{H^{s}(\mathbf{R})}<r$ and $\Vert v\Vert
_{H^{s}(\mathbf{R})}<r$, it is straightforward to show that%

\[
\Vert A(v)w_{0}\Vert_{L^{2}(\mathbf{R})}\leqslant C(C_{f},r)\Vert w_{0}%
\Vert_{H^{s}(\mathbf{R})}<C(C_{f},r)r=k.
\]
where
\[
C_{f}=\underset{|s|\leq\Vert u_{c}\Vert_{L^{\infty}(\mathbf{R})}+r}{\max
}(|f^{\prime}(s)|+|f^{\prime\prime}(s)|).
\]
Thus (C1) holds. \newline

Note that $A_{0}$ is skew-adjoint, therefore one has $\langle A_{0}%
w,w\rangle=0$. Also, it is easy to check that
\begin{align*}
\;\;\;\;\langle A_{1}(v)w,w\rangle\newline &  =\int f^{\prime}(u_{c}%
+v)(\partial_{x}w)wdx+\int\partial_{x}u_{c}\int_{0}^{1}f^{\prime\prime}%
(u_{c}+\tau v)wd\tau wdx\newline\\
&  =-\frac{1}{2}\int f^{\prime\prime}(u_{c}+v)\partial_{x}(u_{c}%
+v)w^{2}dx+\int\partial_{x}u_{c}\int_{0}^{1}f^{\prime\prime}(u_{c}+\tau
v)wd\tau wdx\newline\\
&  \geq-(\frac{1}{2}\Vert f^{\prime\prime}(u_{c}+v)\partial_{x}(u_{c}%
+v)\Vert_{L^{\infty}}+\Vert\partial_{x}u_{c}\Vert_{L^{\infty}}\Vert
f^{\prime\prime}(u_{c}+\tau v)\Vert_{L^{\infty}})\Vert w\Vert_{L^{2}}^{2}.
\end{align*}
Therefore $A(v)\ $generates a $C_{0}\ $semigroup from $L^{2}(\mathbf{R})\ $to
$L^{2}(\mathbf{R})\ $and $A(v)\ $is stable for $v\in B_{r}$.

Following the similar argument as in the proof of Lemma 5.5 in \cite{pazy},
one can verify (C3) by letting $S=\Lambda^{s}$, where $\Lambda^{s}\ $is an
operator with Fourier symbol $(1+\xi^{2})^{s/2}$. We only consider $s=2n$,
where $n$ is any positive integer. It is easy to check that
\[
(1-\partial_{x}^{2})^{s/2}=(1-\partial_{x}^{2})^{n}=\sum_{k=0}^{n}C_{k}%
^{n}(-\partial_{x}^{2})^{k},
\]
where $C_{k}^{n}$ is the number of $k$-combinations.

Then one can check that
\[%
\begin{split}
& \Lambda^{2n}(f^{\prime}(u_{c}+v)\Lambda^{-2n}\partial_{x}w)- f^{\prime
}(u_{c}+v)\partial_{x}w\\
=  & \sum_{k=1}^{n}C_{k}^{n}(-\partial^{2}_{x})^{k}(f^{\prime}(u_{c}%
+v)\Lambda^{-2n}\partial_{x}w)- f^{\prime}(u_{c}+v)\sum_{k=1}^{n}C_{k}%
^{n}(-\partial^{2}_{x})^{k}(\Lambda^{-2n}\partial_{x}w).
\end{split}
\]

It follows that
\[
\Vert\Lambda^{2n}f^{\prime}(u_{c}+v)\Lambda^{-2n}\partial_{x}w- f^{\prime
}(u_{c}+v)\partial_{x}w \Vert_{L^{2}(\mathbf{R})}\leq C(C_{f,n})\Vert w
\Vert_{L^{2}(\mathbf{R})},
\]
where
\[
C_{f,2n}=\underset{|s|\leq\Vert u_{c}\Vert_{L^{\infty}(\mathbf{R})}+r}{\max
}(|f^{\prime}(s)|+|f^{\prime\prime}(s)|+\cdots|f^{(2n+1)}(s)|).
\]

Moreover, it is easy to check that
\[%
\begin{split}
& \Vert\Lambda^{2n}(\partial_{x}u_{c}\int_{0} ^{1}f^{\prime\prime}(u_{c}+\tau
v)\Lambda^{-2n}wd\tau)-\partial_{x}u_{c}\int_{0} ^{1}f^{\prime\prime}%
(u_{c}+\tau v)wd\tau\Vert_{L^{2}(\mathbf{R})}\\
\leq & C(C_{f^{\prime},2n}, \Vert u_{c} \Vert_{W^{2n+1,\infty}})\Vert w
\Vert_{L^{2}(\mathbf{R})}.
\end{split}
\]

Thus, (C3) holds. It is trivial to verify (C4). So we complete the proof of
this lemma.
\end{proof}

Now we are ready to show nonlinear localized instability. Let
\begin{equation}
U_{1}(t,x)=2\operatorname{Re}\int_{I}v_{1}(k,x)e^{\lambda(k)t}e^{ikx}%
\,dk,\quad\quad(t,x)\in\mathbf{R}^{+}\times\mathbf{R}, \label{U1loc}%
\end{equation}
It is easy to see that $U_{1}\left(  t,x\right)  \ $is a real-valued solution
to (\ref{kdvln}) with initial data $U_{1}\left(  0,x\right)  =u_{1}\left(
x\right)  $ (defined in (\ref{initial data-line})). Denote $\lambda
_{0}=\operatorname{Re}\lambda\left(  k_{0}\right)  $.

\begin{lemma}
\label{Lemma-U1-estimates} There exist $c_{1}>c_{2}>0$ such that
\begin{equation}
\frac{c_{2}}{(1+t)^{\frac{1}{l}}}e^{\lambda_{0}t}\leq\Vert U_{1}%
(t,x)\Vert_{L^{2}(\mathbf{R})}\leq\frac{c_{1}}{(1+t)^{\frac{1}{l}}}%
e^{\lambda_{0}t},\ \ t\geq0. \label{estimate-U-1}%
\end{equation}

\end{lemma}

\begin{proof}
By Lemma \ref{ptor}, we have
\[
\Vert U_{1}(t,x)\Vert_{L^{2}(\mathbf{R})}^{2}\thickapprox\int_{I}\Vert
v_{1}(k,x)\Vert_{L_{x}^{2}(\mathbb{T}_{2\pi})}^{2}e^{\operatorname{Re}%
\lambda(k)t}\,dk\thickapprox\int_{I}e^{\operatorname{Re}\lambda(k)t}\,dk.
\]
Denote $I=\left[  k_{0}-\eta,k_{0}\right]  $, $\eta>0$. By
(\ref{taylor-real-lambda}), when $\eta$ is small enough, for$\ $any $k\in I,$
we have
\[
-2a_{0}\left(  k-k_{0}\right)  ^{l}\leq\operatorname{Re}\lambda\left(
k\right)  -\operatorname{Re}\lambda\left(  k_{0}\right)  \leq-\frac{1}{2}%
a_{0}\left(  k-k_{0}\right)  ^{l}.
\]
So letting $k_{1}=k-k_{0}$, then
\[
e^{\lambda_{0}t}\int_{-\eta}^{0}e^{-2a_{0}k_{1}^{l}t}\,dk_{1}\leq\int%
_{I}e^{\operatorname{Re}\lambda(k)t}\,dk\leq e^{\lambda_{0}t}\int_{-\eta}%
^{0}e^{-\frac{1}{2}a_{0}k_{1}^{l}t}\,dk_{1}%
\]
When $0\leq t\leq1$, it is easy to estimate that
\[
\int_{-\eta}^{0}e^{-\frac{1}{2}a_{0}k_{1}^{l}t}\,dk_{1}\leq\eta,\ \ \int%
_{-\eta}^{0}e^{-2a_{0}k_{1}^{l}t}\,dk_{1}\geq e^{-2a_{0}\eta^{l}}\eta.
\]
When $t>1$, by direct calculations we have
\begin{equation}
\int_{-\eta}^{0}e^{-\frac{1}{2}a_{0}k_{1}^{l}t}\,dk_{1}=\frac{1}{t^{\frac
{1}{l}}}\int_{0}^{\eta^{l}t}{\frac{p^{\frac{1}{l}-1}e^{-\frac{1}{2}a_{0}p}}%
{l}}\,dp\leq\frac{c_{0}}{t^{\frac{1}{l}}}, \label{1/m}%
\end{equation}
where
\[
c_{0}=\frac{1}{l}\int_{0}^{+\infty}p^{\frac{1}{l}-1}e^{-\frac{1}{2}a_{0}%
p}\,dp<\infty.
\]
Similarly,
\[
\int_{-\eta}^{0}e^{-\frac{1}{2}a_{0}k_{1}^{l}t}\,dk_{1}\geq\frac{c_{0}%
^{\prime}}{t^{\frac{1}{l}}},\ \ c_{0}^{\prime}=\frac{1}{l}\int_{0}^{\eta^{l}%
}p^{\frac{1}{l}-1}e^{-2a_{0}p}\,dp.
\]
Combining above, we get the estimate (\ref{estimate-U-1}).
\end{proof}

\begin{proof}
[Proof of Theorem \ref{thm-smooth-f} ii)]Following the same way as in the
periodic case, we construct an approximate solution $U^{app}$ to (\ref{kdvtv})
of the form
\begin{equation}
U^{app}=u_{c}+\sum_{j=1}^{N}\delta^{j}U_{j}, \label{uapp-line}%
\end{equation}
where $U_{1}$ is defined in (\ref{U1loc}). By Lemma \ref{Lemma-U1-estimates}
\[
\Vert U_{1}(t,x)\Vert_{H^{s}(\mathbf{R})}\lesssim C(s)\frac{e^{\lambda_{0}t}%
}{(1+t)^{\frac{1}{l}}}.
\]
Following the same arguments as in the proof of Theorem \ref{thm-smooth-f} i),
for $j=1,2,\cdots,N$, we solve $U_{j}$ by the equation
\[
\partial_{t}U_{j}=JLU_{j}+\partial_{x}P_{j}(U_{1},U_{2},\cdots,U_{j-1}%
),\ U_{j}|_{t=0}=0.
\]
By Lemma \ref{lemma-semigroup-inhomo-line}, we obtain
\begin{equation}
\Vert U_{j}(t,x)\Vert_{H^{s}{(\mathbf{R})}}\leqslant C_{j}\left(
\frac{e^{\lambda_{0}t}}{(1+t)^{\frac{1}{l}}}\right)  ^{j}. \label{ujsup}%
\end{equation}
Define $T_{\delta}$ by the equation
\[
\frac{\delta e^{\lambda_{0}T_{\delta}}}{(1+T_{\delta})^{\frac{1}{l}}}=\theta,
\]
where $\theta$ is to be determined. Then $T_{\delta}=O\left(  \left\vert
\ln\delta\right\vert \right)  $. The energy estimate in Lemma
\ref{lemma-energy-estimate} is still true in $H^{2}\left(  \mathbf{R}\right)
$. Let $U_{\delta}\left(  x,t\right)  $ be the solution of (\ref{kdvtv}) with
initial data $U_{\delta}\left(  x,0\right)  =u_{c}+\delta u_{1}\left(
x\right)  $. Then by the same arguments as in the periodic case, when $\theta$
is small enough, we have
\begin{align*}
&  \left\Vert U_{\delta}\left(  T^{\delta},x\right)  -u_{c}\left(  x\right)
\right\Vert _{L^{2}\left(  R\right)  }\\
\geq & C_{1}\frac{\delta e^{\operatorname{Re}\lambda T^{\delta}}}%
{(1+T_{\delta})^{\frac{1}{l}}}-C_{2}\left(  \frac{\delta e^{\operatorname{Re}%
\lambda T^{\delta}}}{(1+T_{\delta})^{\frac{1}{l}}}\right)  ^{2}=C_{1}%
\theta-C_{2}\theta^{2}\\
\geq & \frac{1}{2}C_{1}\theta\text{.}%
\end{align*}
This proves the nonlinear instability in the localized space.
\end{proof}

\section{Nonlinear instability by bootstrap arguments}

\label{section-bootstrap}

The proof of nonlinear instability by constructing higher order approximate
solutions requires the nonlinear term $f\left(  s\right)  $ in (\ref{kdvtv})
to be in $C^{\infty}(\mathbf{R})$. In this section, we give a different proof
by using bootstrap arguments, for the case when $f$ is not smooth. We assume
that (\ref{kdvtv}) is locally well-posed in the energy space $H^{\frac{m}{2}}%
$, which is certainly satisfied under the assumption (\ref{assumption-f}) (see
Lemma \ref{lemma-wp}). We will prove nonlinear instability for the nonlinear
term $f\in C^{1}(\mathbf{R})$ with the growth conditions
(\ref{assumption-f-bootstrap}) and (\ref{assumption-F-bootstrap}). The
bootstrap arguments are done in three steps. First, we use the energy
conservation to control the growth of the energy norm in $H^{\frac{m}{2}}$
from the assumed $L^{2}$ growth. Then we use the semigroup estimates in
$H^{-1}$ to control the growth of $H^{-1}$ norm of the nonlinear part of the
solution. Lastly, the estimates are closed by using the interpolation of
$L^{2}$ by $H^{\frac{m}{2}}$ and $H^{-1}$.

\begin{proof}
[Proof of Theorem \ref{thm-bootstrap-instability}]We only give the proof for
localized perturbations since it is similar for multiple periodic perturbations.

Step 1. (bootstrap from $L^{2}$ to $H^{\frac{m}{2}}$).

The nonlinear equation for the perturbation $u$ of $u_{c}\ $in the traveling
frame $\left(  x-ct,t\right)  $ is
\begin{equation}
\partial_{t}u-JLu+\partial_{x}\left(  f\left(  u+u_{c}\right)  -f\left(
u_{c}\right)  -f^{\prime}\left(  u_{c}\right)  u\right)  =0,
\label{eqn-localized-bootstrap}%
\end{equation}
where $J,L$ are defined in (\ref{definition-J-L}). For any $\delta>0$, we
choose the initial data $u_{\delta}\left(  0\right)  =\delta u_{1}$, where
$u_{1}$ is defined in (\ref{initial data-line}). Then by Lemma
\ref{Lemma-U1-estimates},
\[
\frac{C_{0}\delta e^{\lambda_{0}t}}{(1+t)^{\frac{1}{l}}}\leq\Vert
e^{tJL}u_{\delta}\left(  0\right)  \Vert_{L^{2}(\mathbf{R})}\leq\frac
{C_{1}\delta e^{\lambda_{0}t}}{(1+t)^{\frac{1}{l}}},
\]
for some $C_{0},C_{1}>0$, $l\in\mathbf{N}$, where $\lambda_{0}$ is the largest
growth rate defined in (\ref{defn-largest-lambda-0}). Define $T_{1}>0$ to be
the maximal time such that
\[
\left\Vert u_{\delta}\left(  t\right)  \right\Vert _{L^{2}}\leq\frac
{2C_{1}\delta e^{\lambda_{0}t}}{(1+t)^{\frac{1}{l}}},\ 0\leq t\leq T_{1}.
\]
where $u_{\delta}\left(  t\right)  $ is the solution of
(\ref{eqn-localized-bootstrap}) with the initial data $u_{\delta}\left(
0\right)  $. Define $T_{\delta}$ by
\[
\frac{\delta e^{\lambda_{0}T_{\delta}}}{(1+T_{\delta})^{\frac{1}{l}}}=\theta,
\]
where $\theta>0$ is to be determined. We will show that $T_{1}>T_{\delta}$
when $\theta$ is small. Suppose otherwise $T_{1}\leq T_{\delta}$. The equation
(\ref{eqn-localized-bootstrap}) has the conserved energy-momentum functional
\[
H\left(  u\right)  =\frac{1}{2}\left\langle Lu,u\right\rangle -\int%
_{\mathbf{R}}\left(  F\left(  u+u_{c}\right)  -F\left(  u_{c}\right)
-f\left(  u_{c}\right)  u-\frac{1}{2}f^{\prime}\left(  u_{c}\right)
u^{2}\right)  dx,
\]
since (\ref{eqn-localized-bootstrap}) can be written in the Hamiltonian form
$\partial_{t}u=\partial_{x}H^{\prime}\left(  u\right)  $. By the assumption
(\ref{assumption-symbol-bootstrap}), there exists $c_{0}>0$ such that%
\[
\left\langle \mathcal{M}u,u\right\rangle \geq c_{0}\left\Vert u\right\Vert
_{H^{\frac{m}{2}}}^{2},\ \text{for any }u\in H^{\frac{m}{2}}\text{. }%
\]
Let $T_{2}$ be the maximal time such that
\begin{equation}
\left\Vert u_{\delta}\left(  t\right)  \right\Vert _{H^{\frac{m}{2}}}\leq
\frac{C_{2}\delta e^{\lambda_{0}t}}{(1+t)^{\frac{1}{l}}},\ 0\leq t\leq T_{2},
\label{bootstrap-energy}%
\end{equation}
where
\[
C_{2}=\frac{2}{\sqrt{c_{0}}}\left(  8\left\vert c+f^{\prime}\left(
u_{c}\right)  \right\vert _{\infty}C_{1}^{2}+\frac{3\left\vert \left\langle
Lu_{1},u_{1}\right\rangle \right\vert }{a_{0}^{2}}\right)  ^{\frac{1}{2}},
\]
with
\[
a_{0}=\min_{t\geq0}\frac{e^{\lambda_{0}t}}{(1+t)^{\frac{1}{l}}}>0\text{. }%
\]
We claim that $T_{2}>T_{1}$. Suppose otherwise $T_{2}\leq T_{1}$. Then by the
energy conservation $H\left(  u_{\delta}\left(  t\right)  \right)  =H\left(
u_{\delta}\left(  0\right)  \right)  $ and the assumption
(\ref{assumption-F-bootstrap}), we have
\begin{align}
c_{0}\left\Vert u_{\delta}\left(  t\right)  \right\Vert _{H^{\frac{m}{2}}%
}^{2}  &  \leq\left\langle \mathcal{M}u_{\delta}\left(  t\right)  ,u_{\delta
}\left(  t\right)  \right\rangle \label{estimate-energy}\\
&  \leq\left\vert c+f^{\prime}\left(  u_{c}\right)  \right\vert _{\infty
}\left\Vert u_{\delta}\left(  t\right)  \right\Vert _{L^{2}}^{2}+\left\langle
Lu_{\delta}\left(  0\right)  ,u_{\delta}\left(  0\right)  \right\rangle
\nonumber\\
&  \ \ \ \ \ +O\left(  \left\Vert u_{\delta}\left(  t\right)  \right\Vert
_{H^{\frac{m}{2}}}^{p_{2}}+\left\Vert u_{\delta}\left(  0\right)  \right\Vert
_{H^{\frac{m}{2}}}^{p_{2}}\right)  ,\nonumber
\end{align}
for any $0\leq t\leq T_{2}$. Here, we use the fact that $L^{P_{2}}\left(
\mathbf{R}\right)  \hookrightarrow H^{\frac{m}{2}}\left(  \mathbf{R}\right)  $
when $\frac{m}{2}\geq\frac{1}{2}$. For any $t\leq T_{2}\leq T_{1}\leq
T_{\delta}$, by (\ref{bootstrap-energy}) we have
\[
\left\Vert u_{\delta}\left(  t\right)  \right\Vert _{H^{\frac{m}{2}}}\leq
\frac{C_{2}\delta e^{\lambda_{0}t}}{(1+t)^{\frac{1}{l}}}\leq\frac{C_{2}\delta
e^{\lambda_{0}T_{\delta}}}{(1+T_{\delta})^{\frac{1}{l}}}=C_{2}\theta\text{. }%
\]
Therefore (\ref{estimate-energy}) implies that for $0\leq t\leq T_{2}$, we
have
\begin{align*}
c_{0}\left\Vert u_{\delta}\left(  t\right)  \right\Vert _{H^{\frac{m}{2}}%
}^{2}  &  \leq\left\vert c+f^{\prime}\left(  u_{c}\right)  \right\vert
_{\infty}\left(  \frac{2C_{1}\delta e^{\lambda_{0}t}}{(1+t)^{\frac{1}{l}}%
}\right)  ^{2}+\delta^{2}\left\vert \left\langle Lu_{1},u_{1}\right\rangle
\right\vert \\
\ \ \ \  &  \ \ \ \ \ +C^{\prime}C_{2}\theta^{p_{2}-2}\left(  \left\Vert
u_{\delta}\left(  t\right)  \right\Vert _{H^{\frac{m}{2}}}^{2}+\delta
^{2}\left\Vert u_{1}\right\Vert _{H^{\frac{m}{2}}}^{2}\right)  ,
\end{align*}
and thus by choosing $\theta$ small enough
\begin{align*}
\left\Vert u_{\delta}\left(  t\right)  \right\Vert _{H^{\frac{m}{2}}}^{2}  &
\leq\frac{1}{c_{0}}\left(  2\left\vert c+f^{\prime}\left(  u_{c}\right)
\right\vert _{\infty}\left(  \frac{2C_{1}\delta e^{\lambda_{0}t}}%
{(1+t)^{\frac{1}{l}}}\right)  ^{2}+3\delta^{2}\left\vert \left\langle
Lu_{1},u_{1}\right\rangle \right\vert \right) \\
&  \leq\frac{1}{c_{0}}\left(  8\left\vert c+f^{\prime}\left(  u_{c}\right)
\right\vert _{\infty}C_{1}^{2}+\frac{3\left\vert \left\langle Lu_{1}%
,u_{1}\right\rangle \right\vert }{a_{0}^{2}}\right)  \left(  \frac{\delta
e^{\lambda_{0}t}}{(1+t)^{\frac{1}{l}}}\right)  ^{2}\\
&  =\frac{1}{4}C_{2}^{2}\left(  \frac{\delta e^{\lambda_{0}t}}{(1+t)^{\frac
{1}{l}}}\right)  ^{2},\ \
\end{align*}
for $0\leq t\leq T_{2}.$This is in contradiction to the definition of $C_{2}$
and shows that $T_{2}>T_{1}$.

Step 2 (bootstrap from $L^{2}$ to $H^{-1}$).

The solution $u_{\delta}\left(  t\right)  $ to (\ref{eqn-localized-bootstrap})
can be written as
\begin{align*}
u_{\delta}\left(  t\right)   &  =e^{tJL}u_{\delta}\left(  0\right)  -\int%
_{0}^{t}e^{\left(  t-s\right)  JL}\partial_{x}\left(  f\left(  u_{\delta
}\left(  s\right)  +u_{c}\right)  -f\left(  u_{c}\right)  -f^{\prime}\left(
u_{c}\right)  u_{\delta}\left(  s\right)  \right)  ds\\
&  =u_{l}\left(  t\right)  +u_{n}\left(  t\right)  .
\end{align*}
By (\ref{bootstrap-energy}), Lemma \ref{lemma-semigroup-H-1} and the
assumption (\ref{assumption-f-bootstrap}), when $0\leq t\leq T_{1}$ we have
\begin{align*}
\left\Vert u_{n}\left(  t\right)  \right\Vert _{H^{-1}}  &  \lesssim\int%
_{0}^{t}\left\Vert e^{\left(  t-s\right)  JL}\right\Vert _{H^{-1}}\left\Vert
f\left(  u_{\delta}\left(  s\right)  +u_{c}\right)  -f\left(  u_{c}\right)
-f^{\prime}\left(  u_{c}\right)  u_{\delta}\left(  s\right)  \right\Vert
_{L^{2}}ds\\
&  \lesssim\int_{0}^{t}C(\varepsilon)e^{\left(  \lambda_{0}+\varepsilon
\right)  \left(  t-s\right)  }\left\Vert u_{\delta}\left(  s\right)
\right\Vert _{H^{\frac{m}{2}}}^{p_{1}}ds\\
&  \leq\int_{0}^{t}C(\varepsilon)e^{\left(  \lambda_{0}+\varepsilon\right)
\left(  t-s\right)  }\left(  \frac{C_{2}\delta e^{\lambda_{0}s}}%
{(1+s)^{\frac{1}{l}}}\right)  ^{p_{1}}ds\\
&  \lesssim\left(  \frac{C_{2}\delta e^{\lambda_{0}t}}{(1+t)^{\frac{1}{l}}%
}\right)  ^{p_{1}},
\end{align*}
by choosing $\varepsilon<\left(  p_{1}-1\right)  \lambda_{0}$ and using Lemma
\ref{lemma-semigroup-inhomo-line}.

Step 3 (Interpolation and closing of the estimates).

For $0\leq t\leq T_{1},$ by interpolation we have
\begin{align}
\left\Vert u_{n}\left(  t\right)  \right\Vert _{L^{2}} &  \leq\left\Vert
u_{n}\left(  t\right)  \right\Vert _{H^{-1}}^{\alpha_{1}}\left\Vert
u_{n}\left(  t\right)  \right\Vert _{H^{\frac{m}{2}}}^{1-\alpha_{1}%
}\ \ \ \left(  \alpha_{1}=\frac{m}{m+2}\right)  \label{estimate-higher-L-2}\\
&  \lesssim\left(  \frac{\delta e^{\lambda_{0}t}}{(1+t)^{\frac{1}{l}}}\right)
^{\alpha p_{1}+1-\alpha_{1}},\nonumber
\end{align}
where we use
\[
\left\Vert u_{n}\left(  t\right)  \right\Vert _{H^{\frac{m}{2}}}\leq\left\Vert
u_{\delta}\left(  t\right)  \right\Vert _{H^{\frac{m}{2}}}-\left\Vert
u_{l}\left(  t\right)  \right\Vert _{H^{\frac{m}{2}}}\lesssim\frac{\delta
e^{\lambda_{0}t}}{(1+t)^{\frac{1}{l}}}.
\]
Noticing that $p_{3}=\alpha p_{1}+1-\alpha_{1}>1$, so when $0\leq t\leq
T_{1}\leq T_{\delta}\ $we have
\begin{align*}
\left\Vert u_{\delta}\left(  t\right)  \right\Vert _{L^{2}} &  \leq\left\Vert
u_{l}\left(  t\right)  \right\Vert _{L^{2}}+\left\Vert u_{n}\left(  t\right)
\right\Vert _{L^{2}}\\
&  \leq C_{1}\frac{\delta e^{\lambda_{0}t}}{(1+t)^{\frac{1}{l}}}+C^{\prime
}\left(  \frac{\delta e^{\lambda_{0}t}}{(1+t)^{\frac{1}{l}}}\right)  ^{p_{3}%
}\\
&  \leq\left(  C_{1}+C^{\prime}\theta^{p_{3}-1}\right)  \frac{\delta
e^{\lambda_{0}t}}{(1+t)^{\frac{1}{l}}}<2C_{1}\frac{\delta e^{\lambda_{0}t}%
}{(1+t)^{\frac{1}{l}}},
\end{align*}
by choosing $\theta$ to be small enough. This is in contradiction to the
definition of $T_{1}$. Thus we must have $T_{1}>T$. At $t=T_{\delta}$, by
using (\ref{estimate-higher-L-2}) we have
\begin{align*}
\left\Vert u_{\delta}\left(  T_{\delta}\right)  \right\Vert _{L^{2}} &
\geq\left\Vert u_{l}\left(  T_{\delta}\right)  \right\Vert _{L^{2}}-\left\Vert
u_{n}\left(  T_{\delta}\right)  \right\Vert _{L^{2}}\\
&  \geq C_{0}\frac{\delta e^{\lambda_{0}T_{\delta}}}{(1+T_{\delta})^{\frac
{1}{l}}}-C^{\prime}\left(  \frac{\delta e^{\lambda_{0}T_{\delta}}%
}{(1+T_{\delta})^{\frac{1}{l}}}\right)  ^{p_{3}}\\
&  =C_{0}\theta-C^{\prime}\theta^{p_{3}}\geq\frac{1}{2}C_{0}\theta,
\end{align*}
when $\theta$ is chosen to be small. This finishes the proof of nonlinear
instability for localized perturbations.
\end{proof}

\begin{remark}
The assumption (\ref{assumption-symbol-bootstrap}) could be weakened to
$0<m<1$ depending on the nonlinearity. In the proof, we only need the
embedding of $L^{p}$ into the energy space $H^{\frac{m}{2}}$, where $p>1$ is
the highest power of the nonlinear term $f\left(  u\right)  $ and its
anti-derivative $F\left(  u\right)  $.
\end{remark}

\section{Semilinear equations}

\label{sectionbbm}

In this section, we consider the nonlinear modulational instability of the
generalized BBM equation
\begin{equation}
(1-\partial_{xx})\partial_{t}u+\partial_{x}(u+f(u))=0. \label{bbm}%
\end{equation}

The BBM equation can be viewed as an ordinary differential equation in
$H^{1}$
\[
\partial_{t}u+(1-\partial_{xx})^{-1}\partial_{x}(u+f(u))=0.
\]
Assume that (\ref{bbm}) admits a $T$-periodic traveling solution
$u_{c}(t,x)=u_{c}(x-ct)$. Writing (\ref{bbm}) in the traveling frame
$u(t,x)=U(t,x-ct)$, we arrive at
\begin{equation}
\partial_{t}U-c\partial_{x}U+(1-\partial_{xx})^{-1}\partial_{x}(U+f(U))=0.
\label{bbmtr}%
\end{equation}
Linearizing (\ref{bbmtr}) at $u_{c}$, we obtain the linearized equation in the
Hamiltonian form
\begin{equation}
\partial_{t}U=JLU, \label{bbmln}%
\end{equation}
where
\begin{equation}
J=(1-\partial_{xx})^{-1}\partial_{x},\ L=c\left(  1-\partial_{xx}\right)
-\left(  1+f^{\prime}\left(  u_{c}\right)  \right)  . \label{defn-J-L-bbm}%
\end{equation}
Assume $T=2\pi$. For any $k\in\left[  0,1\right]  $, define
\[
J_{k}=(1-\left(  \partial_{x}+ik\right)  ^{2})^{-1}\left(  \partial
_{x}+ik\right)  ,\ \ L_{k}=c\left(  1-\left(  \partial_{x}+ik\right)
^{2}\right)  -\left(  1+f^{\prime}\left(  u_{c}\right)  \right)  .
\]
As for the KDV type equations, the linear modulational instability of $u_{c}$
means that $J_{k}L_{k}$ has an unstable eigenvalue for some $k\in\left[
0,1\right]  $. Denote $\lambda_{0}$ to be the maximal growth rate of
$e^{tJ_{k}L_{k}},\ k\in\left[  0,1\right]  $. By the same proof of Lemmas
\ref{lemma-trichotomy} and \ref{lemma-semigroup-local-KDV}, we have the
semigroup estimates for (\ref{bbmln}).

\begin{lemma}
\label{lemma-semigroup-bbm}Suppose $u_{c}$ is modulationally unstable.
Consider the semigroup $e^{tJL}$ associated with the solutions of
(\ref{bbmln}), where $J,L$ are given in (\ref{defn-J-L-bbm}). Then

i) the exponential trichotomy in the sense of (\ref{estimate-stable.unstable})
and (\ref{estimate-center}) holds true in the spaces $H^{s}\left(
{\mathbb{T}_{2\pi q}}\right)  $ $\left(  s\geq1,q\in\mathbb{N}\right)  $.

ii) for every $s\geq1$, $\varepsilon>0\ $there exist $C(s,\varepsilon)>0$ such
that
\[
\Vert e^{tJL}u(x)\Vert_{H^{s}{(\mathbf{R})}}\leqslant C(s,\varepsilon
)e^{\left(  \lambda_{0}+\varepsilon\right)  t}\Vert u(x)\Vert_{H^{s}%
{(\mathbf{R})}},\text{ }\forall t>0,
\]
for any $u\in H^{s}{(\mathbf{R})}$.
\end{lemma}

For (\ref{bbmtr}), there is no loss of derivative in the nonlinear term.
Therefore, we can use the semigroup estimates in Lemma
\ref{lemma-semigroup-bbm} to prove nonlinear modulational instability directly
by ODE arguments. We consider localized perturbations below.

\begin{theorem}
\label{thmbbm} Assume $f\in C^{1}\left(  \mathbf{R}\right)  $ and there exists
$p_{1}>1$, such that%
\begin{equation}
\left\vert f\left(  u+v\right)  -f\left(  v\right)  -f^{\prime}\left(
v\right)  u\right\vert \lesssim C\left(  \left\vert u\right\vert _{\infty
},\left\vert v\right\vert _{\infty}\right)  \left\vert u\right\vert ^{p_{1}}.
\label{growth-f-bbm}%
\end{equation}
Let $u_{c}\left(  x-ct\right)  $ be a traveling wave solution of (\ref{bbm})
which is assumed to be linearly modulationally unstable. Then $u_{c}$ is
nonlinearly unstable under localized perturbations in the following sense:
there exists $\theta_{0}>0$, such that for any $s\in\mathbb{N}$ and
arbitrarily small $\delta>0$, there exists a time $T^{\delta}=O\left(
\left\vert \ln\delta\right\vert \right)  $ and a solution $U_{\delta}(t,x)$ to
(\ref{bbmtr}) satisfying $\Vert U_{\delta}(0,x)-u_{c}(x)\Vert_{H^{s}%
(\mathbf{R})}<\delta\ $and$\ \Vert U_{\delta}(T^{\delta},x)-u_{c}%
(x)\Vert_{L^{2}(\mathbf{R})}\geqslant\theta_{0}.$
\end{theorem}

\begin{proof}
For any $\delta>0$, choose the initial perturbation $u_{\delta}(0)=\delta
u_{1}$, where $u_{1}$ is defined as in (\ref{initial data-line}). Then by the
proof of Lemma \ref{Lemma-U1-estimates},
\[
\frac{C_{0}\delta e^{\lambda_{0}t}}{(1+t)^{\frac{1}{l}}}\leq\Vert
e^{tJL}u_{\delta}\left(  0\right)  \Vert_{H^{k}(\mathbf{R})}\leq\frac
{C_{1}\delta e^{\lambda_{0}t}}{(1+t)^{\frac{1}{l}}},\text{ }k=0,1
\]
for some $C_{0},C_{1}>0,l\in\mathbf{N\ }$and $\lambda_{0}$ is the maximal
growth rate defined before. Let $U_{\delta}(t,x)$ be the solution to
(\ref{bbmtr}) with initial value $u_{c}+\delta u_{\delta}(0)$ and $u_{\delta
}=U_{\delta}-u_{c}$, then $u_{\delta}$ satisfies
\begin{equation}
\partial_{t}u_{\delta}=JLu_{\delta}+g(u_{\delta}),\ \ u_{\delta}(0)=\delta
u_{1},\label{bbmhs}%
\end{equation}
where
\[
g(v)=-(1-\partial_{xx})^{-1}\partial_{x}(f(u_{c}+v)-f(u_{c})-f^{\prime}%
(u_{c})v).
\]
Define $T_{1}>0$ to be the maximal time such that
\[
\left\Vert u_{\delta}\left(  t\right)  \right\Vert _{H^{1}}\leq\frac
{2C_{1}\delta e^{\lambda_{0}t}}{(1+t)^{\frac{1}{l}}},\ 0\leq t\leq T_{1}.
\]
Define $T_{\delta}$ by
\[
\frac{\delta e^{\lambda_{0}T_{\delta}}}{(1+T_{\delta})^{\frac{1}{l}}}=\theta,
\]
where $\theta>0$ is to be determined. We will show $T_{1}>T_{\delta}$. Suppose
otherwise, $T_{1}\leq T_{\delta}$. From (\ref{bbmhs}), we have
\begin{align*}
u_{\delta}(t,x) &  =e^{tJL}u_{\delta}(0)+\int_{0}^{t}e^{JL(t-s)}g(u_{\delta
}\left(  s\right)  )\,ds\\
&  =u_{l}+u_{n}.
\end{align*}
Then when $0\leq t\leq T_{1}\leq T_{\delta}$, by using assumption
(\ref{growth-f-bbm}) we have
\begin{align*}
\left\Vert u_{n}\left(  t\right)  \right\Vert _{H^{1}} &  \lesssim\int_{0}%
^{t}\left\Vert e^{\left(  t-s\right)  JL}\right\Vert _{H^{1}}\left\Vert
f\left(  u_{\delta}\left(  s\right)  +u_{c}\right)  -f\left(  u_{c}\right)
-f^{\prime}\left(  u_{c}\right)  u_{\delta}\left(  s\right)  \right\Vert
_{L^{2}}ds\\
&  \lesssim\int_{0}^{t}C(\varepsilon)e^{\left(  \lambda_{0}+\varepsilon
\right)  \left(  t-s\right)  }\left\Vert u_{\delta}\left(  s\right)
\right\Vert _{H^{1}}^{p_{1}}ds\\
&  \lesssim\int_{0}^{t}C(\varepsilon)e^{\left(  \lambda_{0}+\varepsilon
\right)  \left(  t-s\right)  }\left(  \frac{2C_{1}\delta e^{\lambda_{0}s}%
}{(1+s)^{\frac{1}{l}}}\right)  ^{p_{1}}ds\\
&  \lesssim\left(  \frac{\delta e^{\lambda_{0}t}}{(1+t)^{\frac{1}{l}}}\right)
^{p_{1}},
\end{align*}
by choosing $\varepsilon>0$ small. By the same arguments as in the proof of
Theorem \ref{thm-bootstrap-instability}, this leads to a contradiction with
the definition of $T_{1}$. Therefore, $T_{1}\geq T_{\delta}$ and
\begin{align*}
\left\Vert u_{\delta}\left(  T_{\delta}\right)  \right\Vert _{L^{2}} &
\geq\left\Vert u_{l}\left(  T_{\delta}\right)  \right\Vert _{L^{2}}-\left\Vert
u_{n}\left(  T_{\delta}\right)  \right\Vert _{H^{1}}\\
&  \geq C_{0}\frac{\delta e^{\lambda_{0}T_{\delta}}}{(1+T_{\delta})^{\frac
{1}{l}}}-C^{\prime}\left(  \frac{\delta e^{\lambda_{0}T_{\delta}}%
}{(1+T_{\delta})^{\frac{1}{l}}}\right)  ^{p_{3}}\\
&  =C_{0}\theta-C^{\prime}\theta^{p_{3}}\geq\frac{1}{2}C_{0}\theta,
\end{align*}
when $\theta$ is chosen to be small. This finishes the proof of the Theorem.
\end{proof}

\begin{remark}
For multi-periodic perturbations, following the same arguments, we can prove
the nonlinear modulational orbital instability of the generalized BBM
equation. Moreover, since the generalized BBM equation is an infinite
dimensional ODE in $H^{1}$, one can even construct invariant (stable, unstable
and center) manifolds by the standard theory.
\end{remark}

\section{Applications}

\label{application} In this section, we apply our results to some concrete examples.

\subsection{Whitham equation}

\label{subsection-whitham}

Consider the Whitham equation for surface water waves,
\begin{equation}
\label{wh}\partial_{t} u+ \mathcal{M}\partial_{x} u+ \partial_{x}(u^{2})=0,
\end{equation}
where $\mathcal{M}$ is the Fourier multiplier given by
\[
\widehat{\mathcal{M}f}(\xi)=\sqrt{\frac{\tanh\xi}{\xi}}\widehat{f}(\xi).
\]
It is clear that $\|\mathcal{M}(\cdot)\|_{H^{1/2}}\sim\|\cdot\|_{L^{2}}$. It
is clear that $m(\xi)=\sqrt{\frac{\tanh\xi}{\xi}}$ is real-valued, analytic
and even.

The existence of a periodic traveling wave solutions was shown in
\cite{hur-johnson-whitham}.

\begin{lemma}
[\cite{hur-johnson-whitham}]For each $\kappa>0$ and each $b$ with $|b|$
sufficient small, there exists a family of periodic traveling wave solutions
to (\ref{wh}) taking the form
\[
u_{c}(a,b,\kappa)(x,t)=w(a,b)(\kappa(x-c(\kappa,a,b)t))=:w(\kappa,a,b)(z),
\]
for $a$ with $|a|$ sufficiently small , where $w$ and $c$ depend analytically
upon $\kappa$, $a$, and $b$. Moreover, $w$ is smooth, even, and $2\pi
$-periodic in $z$, and $c$ is even in $a$. Furthermore,
\begin{align*}
&  \;\;\;\;w(\kappa,a,b)(z)\newline\\
&  =w_{0}(\kappa,b)+a\cos z+\frac{1}{2}a^{2}(\frac{1}{m(\kappa)-1}+\frac
{\cos(2z)}{m(\kappa)-m(2\kappa)})+O(a(a^{2}+b^{2}))
\end{align*}
and
\[
c(\kappa,a,b)=c_{0}(\kappa,b)+a^{2}(\frac{1}{m(\kappa)-1}+\frac{1}{2}\frac
{1}{m(\kappa)-m(2\kappa)})+O(a(a^{2}+b^{2}))
\]
as $|a|$, $|b|\rightarrow0$, where
\[
c_{0}(\kappa,b):=m(\kappa)+2b(1-m(\kappa))-6b^{2}(1-m(\kappa))+O(b^{3})
\]
and
\[
w_{0}(\kappa,b):=b(1-m(\kappa))-b^{2}(1-m(\kappa))+O(b^{3}).
\]

\end{lemma}

One can check that%

\begin{align*}
&  c-\Vert f^{\prime}(u_{c})\Vert_{L^{\infty}(\mathbb{T})}\newline\\
&  =c-2\Vert u_{c}\Vert_{L^{\infty}(\mathbb{T})}\\
&  =m(\kappa)+2b(1-m(\kappa))-b(1-m(\kappa))-a\cos z+O(a^{2}+b^{2})\\
&  \geqslant\varepsilon_{0}>0,
\end{align*}
when $|a|$, $|b|\ $are sufficiently small. So the assumption
(\ref{assumption-TW-whitham}) is satisfied.

Moreover, the linear modulational instability of $u_{c}(a,b,\kappa)\ $\ is
shown in \cite{hur-johnson-whitham} for $\kappa>0\ $large enough. Therefore,
we can apply Theorem \ref{thm-smooth-f} to obtain nonlinear modulational
instability of $u_{c}(a,b,\kappa)\ $when $|a|$, $|b|\ $are sufficiently small
and $\kappa>0\ $is sufficiently large.

\subsection{The Nonlinear Schr\"odinger equation}

We consider in this section the focusing NLS equation
\begin{equation}
iu_{t}+u_{xx}+|u|^{2}u=0, \label{nls}%
\end{equation}
in which $x\in\mathbf{R}$, $t\in\mathbf{R}^{+}$, and $u(x,t)\in\mathbb{C}$.
Note that like the generalized BBM equation discussed in Section
\ref{sectionbbm}, the NLS equation is also semi-linear, with no loss of
derivative in the nonlinear term. From the results in
\cite{haragus-gallay-schrodinger}, we know that (\ref{nls}) possesses a family
of small periodic waves of the form $u_{a,b}(x,t)=e^{-it}e^{il_{a,b}x}%
P_{a,b}(k_{a,b}x)$, where
\[
l_{a,b}=\frac{1}{4}(a^{2}-b^{2})+O(a^{4}+b^{4}),
\]%
\[
k_{a,b}=1+\frac{3}{4}(a^{2}+b^{2})+O(a^{4}+b^{4}),
\]%
\[
P_{a,b}(y)=ae^{-iy}+be^{iy}+O(|ab|(|a|+|b|)),
\]
as $(a,b)\rightarrow0$.\newline

In \cite{haragus-gallay-schrodinger}, $u_{a,b}(x,t)$ were written in the form
of
\[
u_{a,b}(x,t)=e^{i(p_{a,b}x-t)}Q_{a,b}(2k_{a,b}x),
\]
and solutions of (\ref{nls}) of the form $u(x,t)=e^{i(p_{a,b}x-t)}%
Q(2k_{a,b}x,t)$ were considered, where
\[
p_{a,b}=l_{a,b}+k_{a,b},Q_{a,b}(z)=e^{-iz/2}P_{a,b}(z/2).
\]
Here $Q_{a,b}(z)$ were claimed to be a member of a two-parameter family of
traveling and rotating waves, see Claim 2 in \cite{haragus-gallay-schrodinger}%
. Moreover, $Q_{a,b}(z)$ were regarded as an equilibrium of a corresponding
evolution equation, and the spectrum of a linear operator at $Q_{a,b}(x)$ was
studied to obtain the linear modulational instability of the small periodic
waves $u_{a,b}(x,t)$. Thus, we can use the same arguments as in Section
\ref{sectionbbm} to prove nonlinear modulational instability of the small
periodic waves $u_{a,b}(x,t)$ as a solution of (\ref{nls}).\newline

\subsection{Fractional KDV-type equation}

Consider the KDV-type equation
\begin{equation}
\partial_{t}u+\partial_{x}(\Lambda^{m}u-u^{p})=0, \label{kdv}%
\end{equation}
where the pseudo differential operator $\Lambda=\sqrt{-\partial_{x}^{2}}$ is
defined by its Fourier multiplier as $\widehat{\Lambda u}(\xi)=|\xi|\hat
{u}(\xi)$. Here we consider $m>\frac{1}{2}$ and either $p\in\mathbb{N}$ or
$p=\frac{q}{n}$ with $q$ and $n$ being even and odd natural numbers, respectively.

It is clear that $\Vert\mathcal{M}(\cdot)\Vert_{L^{2}}\sim\Vert\cdot
\Vert_{H^{m}}$ and $\alpha(\xi)=|\xi|^{m}$ is real-valued and even. \newline

In \cite{johnson-fractional-13}, a family of small periodic traveling waves
$u_{a,b}(t,x)$ of (\ref{kdv}) were constructed for $|a|,|b|<<1$. It was also
showed in Theorem 3.4 of \cite{johnson-fractional-13} that $u_{a,b}(t,x)$ is
linearly modulationally unstable if $m\in(\frac{1}{2},1)$ or if $m>1$ and
$p>p^{\ast}(m)$, where $p^{\ast}(m)$ is defined by
\[
p^{\ast}(m):=\frac{2^{m}(3+m)-4-2m}{2+2^{m}(m-1)}.
\]

Therefore, if $m\in(\frac{1}{2},1)$ and $p\in\mathbb{N\ }$or if $m>1$ and
$p>p^{\ast}(m)$ and $|a|,|b|<<1$, then Theorems \ref{thm-smooth-f} and
\ref{thm-bootstrap-instability} can be applied to obtain nonlinear
modulational instability of $u_{a,b}(t,x)$ for both multiple periodic and
localized perturbations. When $m=2$, equation (\ref{kdv}) is reduced to the
generalized KDV equation.

\subsection{BBM equation}

Consider the BBM equation
\begin{equation}
(1-\partial_{xx})\partial_{t}u+\partial_{x}(u+u^{2})=0. \label{eqn-bbm}%
\end{equation}

In \cite{hur-pandey}, the authors showed that (\ref{eqn-bbm}) admits a family
of periodic traveling wave solutions $u_{c}$ in the following form,%
\begin{align*}
u_{c}(t,x;m,a) &  =a\cos(m(x-ct))+a^{2}\frac{1+m^{2}}{6m^{2}}\cos
(2m(x-ct)-3)+o(a^{3}),\\
c(m,a) &  =\frac{1}{1+m^{2}}-a^{2}\frac{5}{6m^{2}}+o(a^{4}),
\end{align*}
with $|a|\ll1$. Furthermore, it was showed in \cite{hur-pandey} that
$u_{c}(t,x;m,a)$ is linearly modulationally unstable if $m>\sqrt{3}$. Applying
Theorem \ref{thmbbm}, one can obtain the nonlinear modulational instability of
$u_{c}(t,x;m,a)$.

\begin{center}
{\Large Acknowledgement}
\end{center}

Zhiwu Lin is supported in part by a NSF grant DMS-1411803. Shasha Liao is
partially supported by the China Scholarship Council No. 20150620040.


\begin{thebibliography}{99}                                                                                               %


\bibitem {angulo-bona-et-06}Angulo Pava, Jaime; Bona, Jerry L.; Scialom,
Marcia, \textit{Stability of cnoidal waves}, Adv. Differential Equations
\textbf{11} (2006), no. 12, 1321--1374.

\bibitem {angulo-book}Angulo Pava, Jaime, \textit{Nonlinear dispersive
equations. Existence and stability of solitary and periodic travelling wave
solutions}. Mathematical Surveys and Monographs, 156. American Mathematical
Society, Providence, RI, 2009.

\bibitem {guo-bardos02}Bardos, C.; Guo, Y.; Strauss, W.,\textit{ Stable and
unstable ideal plane flows}. Dedicated to the memory of Jacques-Louis Lions.
Chinese Ann. Math. Ser. B\textbf{ 23} (2002), no. 2, 149--164.

\bibitem {benjamin-feir}Benjamin, T. B. and Feir, J. E., \textit{The
disintegration of wave trains on deep water. Part 1. Theory}, J. Fluid Mech.
\textbf{27} (3): 417--437 (1967).

\bibitem {deconick-kdv-spectral}Bottman, Nate; Deconinck, Bernard, \textit{KdV
cnoidal waves are spectrally stable}. Discrete Contin. Dyn. Syst. \textbf{25}
(2009), no. 4, 1163--1180.

\bibitem {hur-et-survey-15}Bronski, Jared C.; Hur, Vera Mikyoung; Johnson,
Mathew, A. \textit{Modulational instability in equations of KdV type.} New
approaches to nonlinear waves, 83--133, Lecture Notes in Phys., 908, Springer,
Cham, 2016.

\bibitem {bronski-hur-14}Bronski, Jared C. and Hur, Vera Mikyoung,
\textit{Modulational instability and variational structure}. Stud. Appl. Math.
\textbf{132} (2014), no. 4, 285--331.

\bibitem {bronski-johnson-arma}Bronski, Jared C. and Johnson, Mathew A.,
\textit{The modulational instability for a generalized Korteweg-de Vries
equation.} Arch. Ration. Mech. Anal. \textbf{197} (2010), no. 2, 357--400.

\bibitem {deconinck-high-frequency}Deconinck, Bernard and Trichtchenko, Olga,
\textit{High-frequency instabilities of small-amplitude solutions of
Hamiltonian PDEs}, preprint, 2015.

\bibitem {solitary-whitham}Ehrnstr\"{o}m, Mats; Groves, Mark D.; Wahl\'{e}n,
Erik, \textit{On the existence and stability of solitary-wave solutions to a
class of evolution equations of Whitham type}. Nonlinearity \textbf{25}
(2012), no. 10, 2903--2936.

\bibitem {haragus-gallay-schrodinger}Gallay, Thierry and H\u{a}r\u{a}gu\c{s},
Mariana, \textit{Stability of small periodic waves for the nonlinear
Schr\"{o}dinger equation}. J. Differential Equations \textbf{234} (2007), no.
2, 544--581.

\bibitem {grenier-2000}Grenier, Emmanuel, \textit{On the nonlinear instability
of Euler and Prandtl equations}. Comm. Pure Appl. Math. \textbf{53} (2000),
no. 9, 1067--1091.

\bibitem {gss90}Grillakis, Manoussos; Shatah, Jalal; Strauss, Walter,
\textit{Stability theory of solitary waves in the presence of symmetry. II.,}
J. Funct. Anal. \textbf{94} (1990), no. 2, 308--348.

\bibitem {guo-strauss95}Guo, Yan; Strauss, Walter A., \textit{Instability of
periodic BGK equilibria}. Comm. Pure Appl. Math. \textbf{48} (1995), no. 8, 861--894.

\bibitem {haragus-kapitula}H\u{a}r\u{a}gu\c{s}, Mariana and Kapitula, Todd,
\textit{On the spectra of periodic waves for infinite-dimensional Hamiltonian
systems.} Phys. D \textbf{237} (2008), no. 20, 2649--2671.

\bibitem {haragus-BBM}H\u{a}r\u{a}gu\c{s}, Mariana, \textit{Stability of
periodic waves for the generalized BBM equation}. Rev. Roumaine Math. Pures
Appl. \textbf{53} (2008), no. 5-6, 445--463.

\bibitem {hur-johnson-whitham}Hur, Vera Mikyoung; Johnson, Mathew A.
\textit{Modulational instability in the Whitham equation for water waves.}
Stud. Appl. Math. \textbf{134 }(2015), no. 1, 120--143.

\bibitem {hur-johnson-stability}Hur, Vera Mikyoung; Johnson, Mathew A.
\textit{Stability of periodic traveling waves for nonlinear dispersive
equations}. SIAM J. Math. Anal. \textbf{47} (2015), no. 5, 3528--3554.

\bibitem {hur-pandey}Hur, Vera Mikyoung; Pandey, Ashish Kumar,
\textit{Modulational instability in nonlinear nonlocal equations of
regularized long wave type}. Phys. D \textbf{325} (2016), 98--112.

\bibitem {hur-pandey-high-frequency}Hur, Vera Mikyoung and Pandey, Ashish
Kumar, \textit{Modulational instability in a full-dispersion shallow water
model}, preprint, 2016.

\bibitem {johnson-fractional-13}Johnson, Mathew A., \textit{Stability of small
periodic waves in fractional KdV-type equations}, SIAM J. Math. Anal.,
\textbf{45} (5) (2013), pp. 3168-3193.

\bibitem {johnson-09}Johnson, Mathew A., \textit{Nonlinear stability of
periodic traveling wave solutions of the generalized Korteweg-de Vries
equation}. SIAM J. Math. Anal. \textbf{41} (2009), no. 5, 1921--1947.

\bibitem {kato-quasilinear}Kato, Tosio, \textit{Quasi-linear equations of
evolution, with applications to partial differential equations}. Spectral
theory and differential equations, pp. 25--70. Lecture Notes in Math., Vol.
448, Springer, Berlin, 1975.

\bibitem {kato-book}Kato, Tosio, \textit{Perturbation theory for linear
operators}. Reprint of the 1980 edition. Classics in Mathematics.
Springer-Verlag, Berlin, 1995.

\bibitem {kato-lecture}Kato, Tosio, \textit{Linear and quasi-linear equations
of evolution of hyperbolic type}. Hyperbolicity, 125--191, C.I.M.E. Summer
Sch., 72, Springer, Heidelberg, 2011.

\bibitem {lighthill}Lighthill, M. J., \textit{Contributions to the theory of
waves in non-linear dispersive systems}. IMA J. Appl. Math. 1, 269--306 (1965).

\bibitem {lin-zeng-Hamiltonian}Lin, Zhiwu and Zeng, Chongchun,
\textit{Instability, index theorem, and exponential trichotomy for Linear
Hamiltonian PDEs}, arXiv:1703.04016.

\bibitem {lin-euler-imrn}Lin, Zhiwu, \textit{Nonlinear instability of ideal
plane flows}. Int. Math. Res. Not. 2004, no. \textbf{41}, 2147--2178.

\bibitem {lin-vm-nonlinear}Lin, Zhiwu; Strauss, Walter, \textit{Nonlinear
stability and instability of relativistic Vlasov-Maxwell systems}. Comm. Pure
Appl. Math. \textbf{60} (2007), no. 6, 789--837.

\bibitem {pazy}Pazy, A., \textit{Semigroups on linear operators and
applications to partial differential equations}, Springer Verlag (1983).

\bibitem {whitham-1967}Whitham, G. B., \textit{Non-linear dispersion of water
waves}, J. Fluid Mech. \textbf{27}:399--412 (1967).

\bibitem {zakharov}Zakharov, V.E., \textit{Stability of periodic waves of
finite amplitude on the surface of a deep fluid}, J. Appl. Mech. Tech. Phys.
\textbf{9} (2) (1968) 190--194.

\bibitem {zakharov-review-2009}Zakharov, V. E. and Ostrovsky, L. A.,
\textit{Modulation instability: the beginning}. Phys. D \textbf{238} (2009),
no. 5, 540--548.
\end{thebibliography}
\end{document}